%% file: pottsManifoldOC.tex
\documentclass[]{article}

\usepackage[utf8]{inputenc}
\usepackage[T1]{fontenc}

\usepackage{amsmath, amsfonts,amsthm}
\usepackage{bbm}

\usepackage[font=footnotesize,labelfont=bf]{caption}

\usepackage{appendix}

\usepackage{color}

\input{symbols}

\input{customcommands}

\newtheorem{theorem}{Theorem}
\theoremstyle{remark}

\usepackage{aliascnt}
\newaliascnt{conj}{theorem}
\newaliascnt{cor}{theorem}
\newaliascnt{lemma}{theorem}
\newaliascnt{prop}{theorem}
\newaliascnt{definition}{theorem}
\newaliascnt{example}{theorem}
\newaliascnt{notation}{theorem}
\newaliascnt{experiment}{theorem}

\theoremstyle{theorem}

\theoremstyle{definition}

\aliascntresetthe{conj}
\aliascntresetthe{cor}
\aliascntresetthe{lemma}
\aliascntresetthe{prop}
\aliascntresetthe{definition}
\aliascntresetthe{example}
\aliascntresetthe{notation}
\aliascntresetthe{experiment}

\newcommand{\prox}{{\rm prox}}
\newcommand{\mean}{{\rm mean}}

\usepackage{snapshot}

\usepackage{booktabs}
\usepackage{multicol}

\usepackage[boxed, ruled, lined]{algorithm2e}

\usepackage[usenames,dvipsnames,svgnames]{xcolor}

\usepackage{enumerate}
\usepackage{csquotes}

\usepackage{subcaption}

\usepackage[pdftex, 
 colorlinks = true,
  linkcolor= black,
  citecolor = black]{hyperref}

\newcommand{\ps}[1]{\left\langle #1 \right\rangle}

\newcommand{\dist}{d}

\newcommand{\Pos}{\ensuremath{\mathrm{Pos}_3}}
\newcommand{\trace}{\mathrm{trace}}

\newcommand{\comment}[1]{}

\usepackage{tikz}
\usepackage{pgfplots}
\usetikzlibrary{spy}

\usepackage[numbers,sort&compress]{natbib}

\title{ 
Mumford-Shah  and Potts Regularization for Manifold-Valued Data
with Applications to DTI and Q-Ball Imaging 
 }
 \author{Andreas Weinmann, Laurent Demaret, Martin Storath\thanks{Andreas Weinmann and Laurent Demaret  are both with the Helmholtz Zentrum München, Germany. Martin Storath is with the Biomedical Imaging Group, École Polytechnique Fédérale de Lausanne, Switzerland.}
 \thanks{
 This work was supported by the German Federal Ministry for Education and Research under SysTec Grant 0315508.
The first author acknowledges support by the Helmholtz Association within the young investigator group \mbox{VH-NG-526}.
The third author was supported by the
European Research Council (ERC) under the European Union's Seventh Framework Programme (FP7/2007-2013) / ERC grant agreement no.~267439. 
 }
 }
\date{\today}

\begin{document}

\newlength\figureheight
\newlength\figurewidth
\setlength\figureheight{0.15\textwidth}

\maketitle

\begin{abstract}
Mumford-Shah and Potts functionals are powerful 
variational models for regularization which are widely used in signal and image processing;
typical applications are edge-preserving denoising and segmentation.
Being both non-smooth and non-convex, they are computationally challenging
even for scalar data.
For manifold-valued data, the problem becomes even more involved since typical features 
of vector spaces are not available.
In this paper, we propose algorithms for Mumford-Shah and for Potts regularization of manifold-valued signals and images.
For the univariate problems, we derive solvers
based on dynamic programming combined  with (convex) optimization techniques for manifold-valued data.
For the class of Cartan-Hadamard manifolds (which includes the data space in diffusion tensor imaging), 
we show
that our algorithms compute
global minimizers for any starting point.
For the multivariate Mumford-Shah and Potts problems (for image regularization) 
we propose a splitting into suitable subproblems which we can solve exactly
using the techniques developed for the corresponding univariate problems.
Our method does not require any a priori restrictions on the edge set 
and we do not have to discretize the data space.
We apply our method to diffusion tensor imaging (DTI) as well as Q-ball imaging.
Using the DTI model, we obtain a segmentation of the corpus callosum.		
\end{abstract}

\section{Introduction}
\input{pottsManifold__Intro}

\input{pottsManifold_Algos}

\input{pottsManifold_DTI}

\section{Application to Q-Balls}
\label{sec:QBall_manifold}

In DTI the diffusion at each pixel/voxel is modeled via a single tensor.
Typically, this tensor has one dominant eigenvalue with corresponding eigenvector pointing to the direction with maximal diffusivity.
This direction is directly related with pathways of, e.g., neural fibers. 
DTI encounters difficulties for modeling voxels with intravoxel directional heterogeneity which, for example, occur at crossings 
of  fiber bundles \cite{alexander2002detection,tuch2002high}.
In order to overcome these limitations, several approaches have been proposed \cite{alexander2002detection,frank2002characterization,ozarslan2003generalized,descoteaux2007regularized}.
One of the most popular among these approaches is Q-ball imaging  \cite{tuch2002high}. 
Here the tensor (seen as an ellipsoid parametrized over a ball) 
is replaced by a more general orientation distribution function (ODF) $\phi: \mathbb{S}^2 \to \R$ 
where $\phi(s)$ essentially corresponds to the diffusivity in direction $s.$  
Since the method allows for more flexibility,
high angular resolution diffusion imaging (HARDI) data (see \cite{tuch2002high,tuch2004q}) are needed.
Further information can be found in the latter references.

\subsection{The Q-ball manifold and the implementation of our algorithm for Q-ball imaging}

In order to derive a Riemannian structure on the Q-ball manifold we follow the approach of \cite{goh2009nonparametric}.
The points in the (discrete) Q-ball manifold are ``square-root parametrized''
(discrete) ODFs which are a kind of samples of continuous ODFs $\phi: \mathbb{S}^2 \to \R$
on a finite subset $S$ of the sphere $\mathbb{S}^2$ with a preferably almost equidistant sampling. To be precise, 
a discrete ODF is a positive function  $\phi: S \to \R$ such that $\sum_{s \in S} \phi^2(s) =1$ (as proposed in \cite{goh2009nonparametric}).
Hence, a discrete ODF can be identified with a point on the sphere $\mathbb{S}^{n-1}.$
Then the set $\Phi $ of all discrete ODFs is the intersection of the positive quadrant with the unit sphere in $\R^n,$ 
and thus can be endowed  with the Riemannian structure inherited from $\mathbb{S}^{n-1}$. 
Then the corresponding metric for the Q-ball manifold is given by 
$$
    d(\phi_1,\phi_2) = \arccos \left(\sum_{s \in S}\phi_1(s) \phi_2(s)\right), \mbox{ for } \phi_1,\phi_2 \in \Phi. 
$$
The basic Riemannian operations have simple closed expressions.
For a point $\phi$ on the unit sphere $\mathbb{S}^{n-1}$ in $\R^n$ and a non-zero tangent vector $v$ to the sphere at $\phi$,
the exponential mapping is given by
$$
   \exp_{\phi} (v) = \phi \cdot \cos \|v\| + \frac{v \cdot \sin \|v\|}{\|v\|},
$$
where $\|\cdot\|$ denotes the Euclidean norm in $\R^n$. The inverse of the exponential mapping is defined for any pair of points 
$\phi_1,\phi_2 \in \Phi$ by 
$$
   \exp^{-1}_{\phi_1}(\phi_2) = d(\phi_1,\phi_2) \cdot \frac{\phi_2- \ps{\phi_1,\phi_2}\phi_1}{\|\phi_2- \ps{\phi_1,\phi_2}\phi_1\|}.
$$
These explicit formulas for the Riemannian $\exp$ mapping and its inverse enable us to directly apply our algorithms for the regularization of 
Q-ball data.

\subsection{Numerical experiments}

We apply our algorithm to synthetic Q-ball data.  
Our examples simulate situations where two fiber bundles intersect. 
In the examples the size of the sampling set on the $2$-sphere is $n=181$ directions.
In order to simulate noisy data,  we use  the method based on the so-called \enquote{soft equator approximation}  
\cite{tuch2004q}. 
We visualize a discrete ODF as a spherical polar plot. 
We compare our results with classical $L^2$-Sobolev 
regularization ($L^2$-$V^2$)
using the cyclic proximal point algorithm of \cite{weinmann2013total}.

Our first example is a univariate signal (Fig.~\ref{fig:QB_ex11_bz_partition}). It contains two kinds of Q-balls: one ``tensor-like'' with a single peak and another one with two peaks.
This illustrative example shows that, also in the Q-ball case, 
our regularization method removes the noise while preserving the jump and its location.

Our second experiment is a Q-ball valued image which simulates the crossing of two fiber bundles (Fig.~\ref{fig:qball2D}). 
Here, we observe that our method removes the noise while preserving the fiber crossing and the directional structures encoded in the Q-balls as well as the edge structure in the image.

\begin{figure}
	 	\def\subfigwidth{0.65\textwidth}
		\def\figurewidth{1\textwidth}
		 \def\hs{\hspace{-0.02\textwidth}}
		 \def\vs{\vspace{0.02\textwidth}}
	\centering
	\hs
	\begin{subfigure}[t]{\subfigwidth}
    \includegraphics[height=\figurewidth, angle=90]{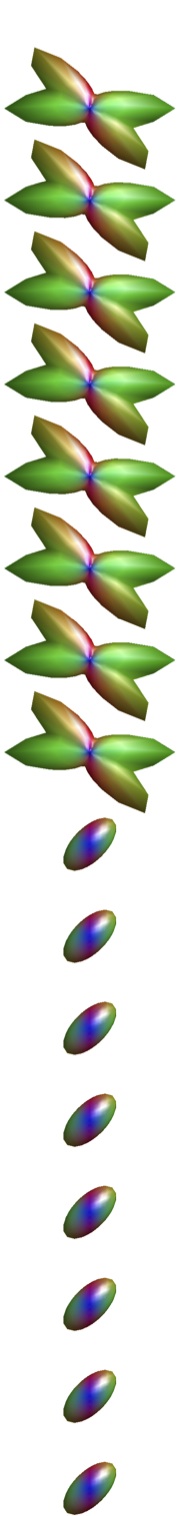}
    \caption{
    }
  \end{subfigure}\\
	\vs
		\hs
	\begin{subfigure}[t]{\subfigwidth}
    \includegraphics[height=\figurewidth, angle=90]{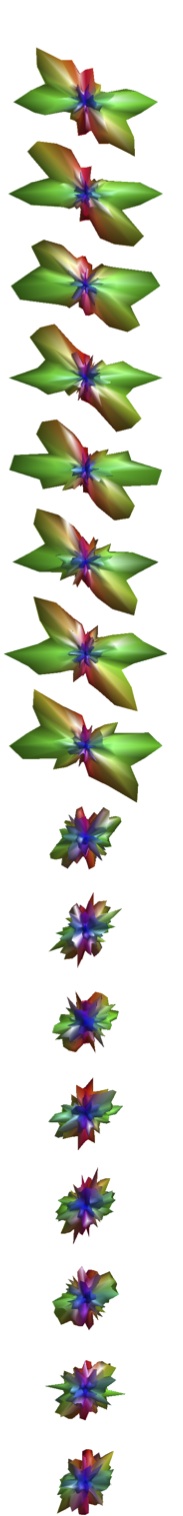}
    \caption{}
  \end{subfigure}\\
  	\vs
	  \hs
  		\begin{subfigure}[t]{\subfigwidth}
		    \includegraphics[height=\figurewidth, angle=90]{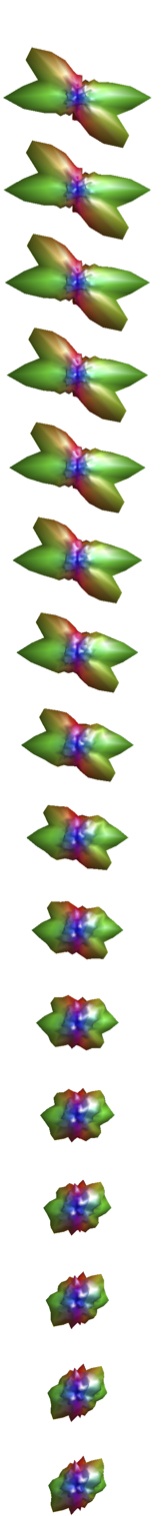}
    \caption{}
  \end{subfigure}\\
  \vs
	\hs
	\begin{subfigure}[t]{\subfigwidth}
    \includegraphics[height=\figurewidth, angle=90]{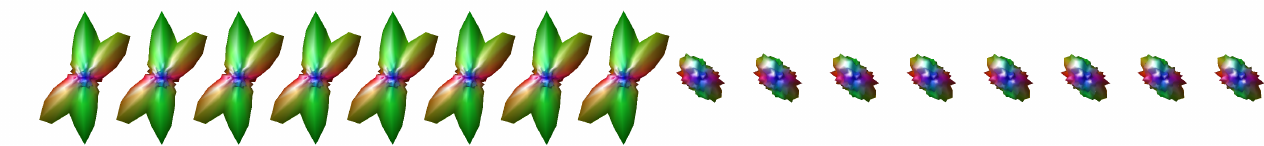}
    \caption{}
  \end{subfigure}
\caption{
	(a) Synthetic piecewise smooth Q-ball signal, 
	(b) noisy data,  
    (c) the manifold analogue of classical Sobolev regularization ($L^2$-$V^2$ with $\alpha=50$),
    (d) Mumford-Shah regularization ($p,q=2$) with parameters $\alpha = 25,\gamma= 0.5$.
    Classical Sobolev regularization removes the noise, but it smoothes out the jump;
    in contrast, the Mumford-Shah regularization removes the noise and preserves the jumps.
    }
   \label{fig:QB_ex11_bz_partition}
\end{figure}

\begin{figure*}
\def\subfigwidth{0.22\textwidth}
		\def\figurewidth{1\textwidth}
				 \def\hs{\hfill}
\tikzstyle{myspy}=[spy using outlines={cyan,dashed,lens={scale=3},width=1.0\textwidth, height=1.0\textwidth, connect spies, 
every spy on node/.append style={thick}}]
\def\figwidth{0.48\columnwidth}
\def\nodeSpy{(-0.45, 0.45)}
\def\nodeWindow{(1.3,-3.0)}
	\begin{subfigure}[t]{\subfigwidth}
	\begin{tikzpicture}[myspy]
	\node {\includegraphics[interpolate=true,width=\figurewidth, trim=45 45 47 55, clip]{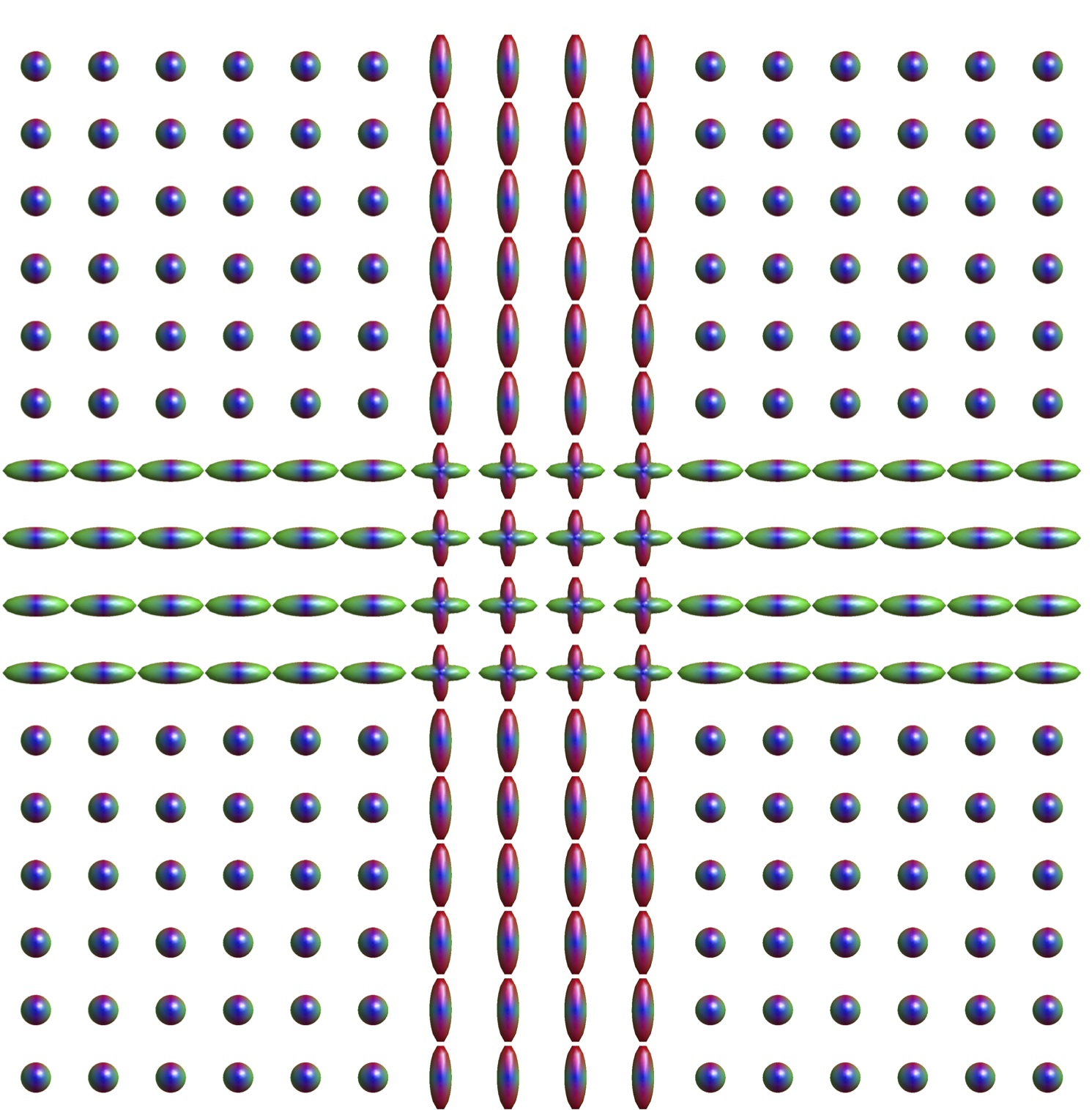}};
	\spy on \nodeSpy in node [left] at \nodeWindow;
\end{tikzpicture}
    \caption{}
  \end{subfigure}
  \hs
  	\begin{subfigure}[t]{\subfigwidth}
  	\begin{tikzpicture}[myspy]
	\node {\includegraphics[interpolate=true,width=\figurewidth, trim=45 45 47 55, clip]{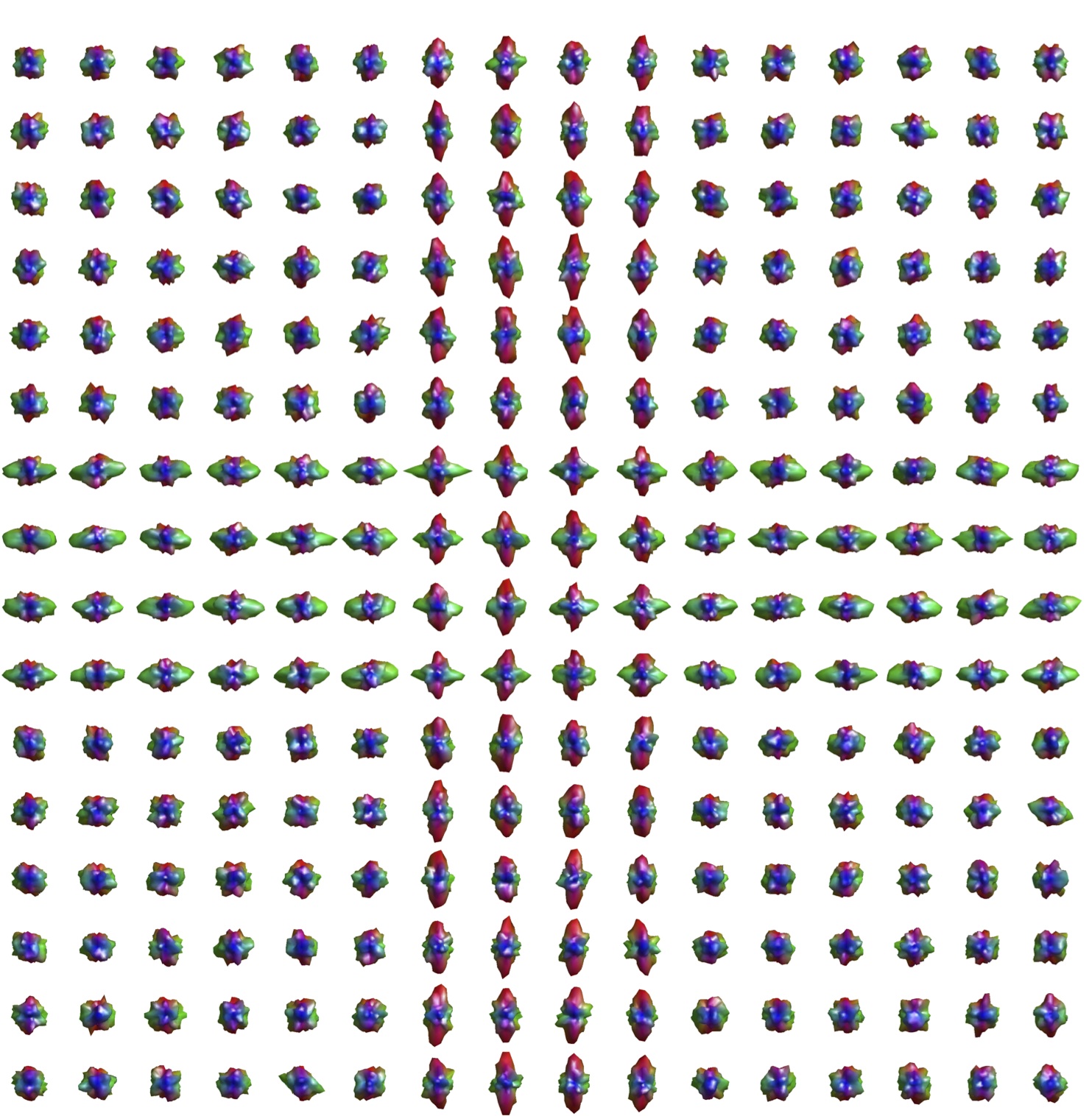}};
	\spy on \nodeSpy in node [left] at \nodeWindow;
\end{tikzpicture}
    \caption{}
  \end{subfigure}
    \hs
  	\begin{subfigure}[t]{\subfigwidth}
	\begin{tikzpicture}[myspy]
	\node {\includegraphics[interpolate=true,width=\figurewidth, trim=45 45 47 55, clip]{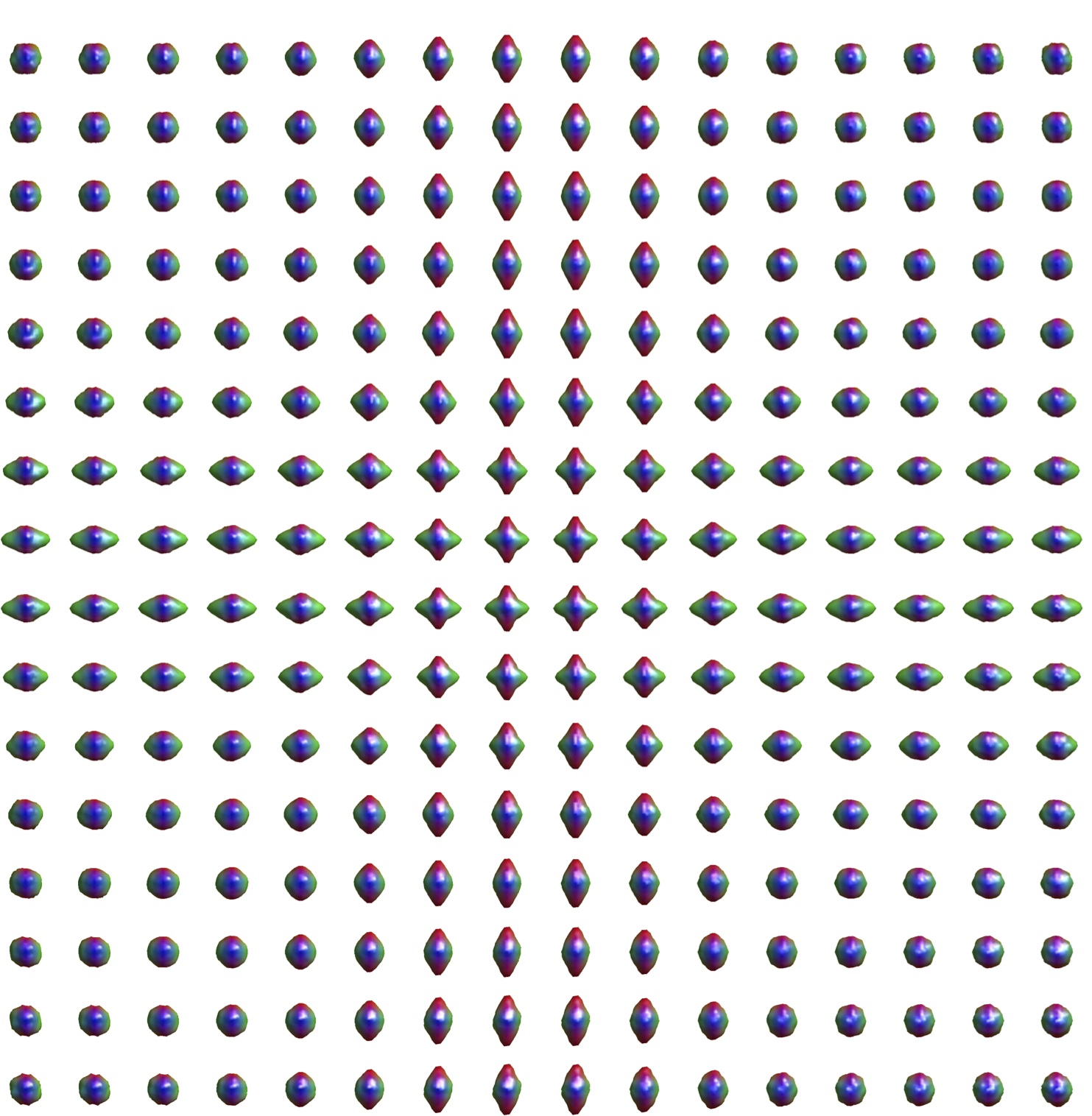}};
	\spy on \nodeSpy in node [left] at \nodeWindow;
\end{tikzpicture}
    \caption{}
  \end{subfigure}
    \hs
  	\begin{subfigure}[t]{\subfigwidth}
	\begin{tikzpicture}[myspy]
	\node {\includegraphics[interpolate=true,width=\figurewidth, trim=45 45 47 55, clip]{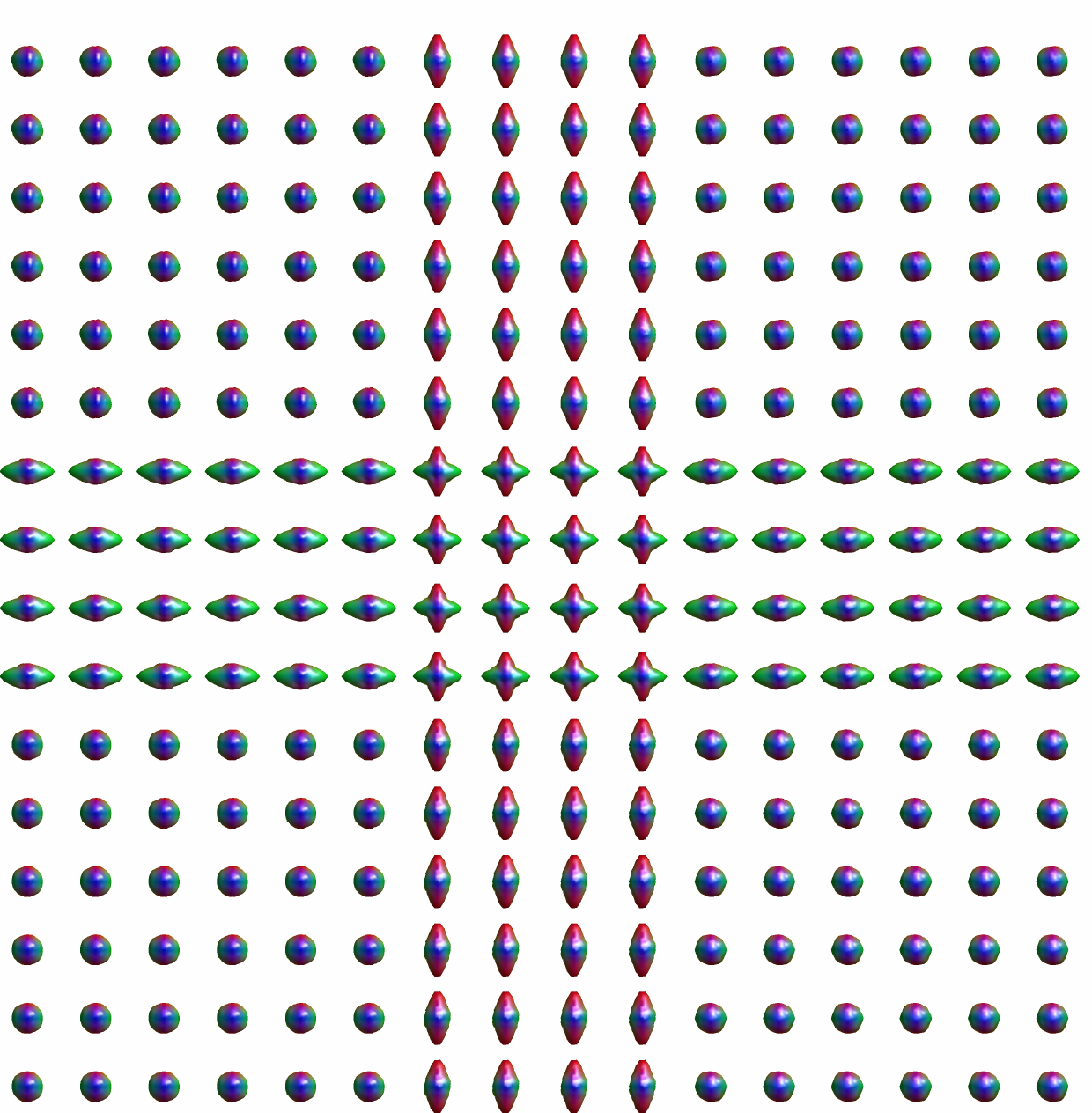}};
	\spy on \nodeSpy in node [left] at \nodeWindow;
\end{tikzpicture}
    \caption{}
  \end{subfigure}
  \caption{
  (a) Synthetic Q-Ball image, 
  (b) noisy data,
  (c) the manifold analogue of classical Sobolev regularization ($L^2$-$V^2$ with parameter $\alpha=1$),
  (d)~Mumford-Shah regularization ($p,q= 2$) with parameters $\alpha=32,\gamma=0.03$. 
  The Sobolev regularization smoothes out the edges and the crossing structures. 
  The Mumford-Shah method recovers the edges as well as the crossings of the original image reliably.
}
\label{fig:qball2D}
\end{figure*}

\section{Conclusion and future research}

In this paper, we proposed algorithms for the non-smooth and non-convex
Mumford-Shah and Potts functionals for manifold-valued signals and images.
We have shown the potential of our method by applying it to DTI and Q-ball imaging.
Using the DTI model, we obtained a segmentation of the corpus callosum.
For signals with values in Cartan-Hadamard manifolds (which includes the data space in diffusion tensor imaging), 
we have seen that our
algorithms for univariate data produce global minimizers for any starting point.
For the Mumford-Shah and Potts problems for image regularization (which is a NP hard problem) 
we have obtained convergence of the proposed splitting approach.

Topics of future research are the application of our algorithms to further nonlinear data spaces
relevant for imaging. Another issue is to build a segmentation pipeline based on our method.
Finally, from a theoretical side, it is interesting to further investigate convergence related questions for general
Riemannian manifolds.

\appendix

\input{pottsManifold_proofs}

{\small
\bibliographystyle{unsrt}
\bibliography{pottsManifoldOC}
}

\end{document}

%% file: symbols.tex
\usepackage{mathrsfs}

\renewcommand{\hat}{\widehat}
\renewcommand{\tilde}{\widetilde}
\renewcommand{\phi}{\varphi}

\renewcommand{\i}{i\,}

\newcommand{\N}{\mathbb{N}}
\newcommand{\Z}{\mathbb{Z}}

\newcommand{\R}{\mathbb{R}}

\newcommand{\Nc}{\mathcal{N}}

%% file: customcommands.tex
\newcommand{\p}[1]{\left( #1 \right)}

\newcommand{\argmin}{\operatorname{\arg}\min}



%% file: pottsManifold__Intro.tex

In their seminal works 
\cite{mumford1985boundary,mumford1989optimal}
Mumford and Shah introduced a  powerful variational approach 
for image regularization.
It consists of the minimization of an energy functional given by 
\begin{equation} \label{eq:mumfordShah}
 \min_{u, C}  \gamma |C| +  \frac{\alpha}{q}  
   \int_{\Omega\setminus C} |D u(x)|^q  dx + \frac{1}{p}\int_{\Omega} d(u(x),f(x))^p dx.
\end{equation}
Here, $f$ represents the data and $u$ is the target variable to optimize for.
In the scalar case, $u$ and $f$ are real-valued functions on a domain $\Omega \subset \R^2,$
$d$ is the Euclidean metric, and $D u$ denotes the gradient (in the weak sense).
In contrast to Tikhonov-type priors, 
 the Mumford-Shah prior penalizes the variation 
 only on the complement of a discontinuity set $C.$ 
Furthermore, the \enquote{length} $|C|$ (i.e., the outer one-dimensional Hausdorff measure) 
 of this discontinuity set is penalized.
 The parameters $\gamma >0$ and $\alpha>0$ control the balance between the 
penalties. 
Basically, the resulting regularization is a smooth approximation to the image $f$
which, at the same time, allows for sharp variations (\enquote{edges}) at the discontinuity set.
The piecewise constant variant of \eqref{eq:mumfordShah} --
often called Potts functional -- corresponds to the degenerate case $\alpha = \infty$
which amounts to removing the second term in \eqref{eq:mumfordShah}.
Typical applications of these functionals
are edge-preserving smoothing and image segmentation.
For further information considering these problems from various perspectives (calculus of variation, stochastics, inverse problems) we exemplarily refer the reader to  
\cite{potts1952some, blake1987visual, geman1984stochastic, ambrosio1990approximation,chambolle1995image, wittich2008complexity,boysen2009consistencies,fornasier2010iterative,fornasier2013existence,jiang2014regularizing} 
and the references therein.  
These references also deal with theoretical questions such as, e.g., the existence of minimizers.
 Mumford-Shah and Potts problems are computationally challenging
 since one has to deal with non-smooth and non-convex functionals. 
Even for scalar data, both problems are NP-hard in dimensions higher than one \cite{veksler1999efficient, boykov2001fast, alexeev2010complexity}.
This makes finding a (global) minimizer infeasible.
However, due to its importance  in image processing,
many  approximative strategies have been proposed  
for scalar- and vector valued data.
Among these are graduated non-convexity \cite{blake1987visual}, approximation by elliptic functionals \cite{ambrosio1990approximation}, 
graph cuts \cite{boykov2001fast}, active contours \cite{tsai2001curve}, 
convex relaxations \cite{pock2009algorithm}, and iterative thresholding approaches \cite{fornasier2010iterative}.

In recent years, regularization of manifold-valued data
has gained a lot of interest. 
For example, sphere-valued data have been considered for SAR imaging \cite{massonnet1998radar}
and non-flat models for color image processing \cite{chan2001total,vese2002numerical,kimmel2002orientation, lai2014splitting}.
Further examples are  $SO(3)$ data expressing vehicle headings, aircraft orientations or  
camera positions \cite{rahman2005multiscale}, and motion group-valued data \cite{rosman2012group}.  
Related work dealing with the processing of manifold-valued data
are 
wavelet-type multiscale transforms 
\cite{rahman2005multiscale,grohs2009interpolatory, weinmann2012interpolatory} and manifold-valued partial differential equations  
\cite{tschumperle2001diffusion, chefd2004regularizing, grohs2013optimal};
statistics on Riemannian manifolds are the topic of \cite{fletcher2012,fletcher2007riemannian,fletcher2004principal,oller1995intrinsic,bhattacharya2003large,bhattacharya2005large,pennec2006intrinsic}. 
In medical imaging, a prominent example with manifold-valued data is diffusion tensor imaging (DTI).
DTI allows to quantify the diffusional characteristics of a specimen non-invasively \cite{basser1994mr, johansen2009diffusion}; see also the overview in \cite{assaf2008diffusion}.
DTI is helpful in the context of neurodegenerative pathologies such as schizophrenia \cite{foong2000neuropathological, kubicki2007review}, autism 
\cite{alexander2007diffusion} or Huntington's disease \cite{rosas2010altered}.
In DTI, the data can be viewed as living in the Riemannian manifold of positive (definite) matrices; 
see, e.g., \cite{pennec2006riemannian}. 
The underlying distance corresponds to the Fisher-Rao metric \cite{radhakrishna1945information}
which is statistically motivated since the positive matrices (called diffusion tensors) represent covariance matrices.
These tensors model the diffusivity of water molecules.
Oriented diffusivity along fiber structures is reflected by the anisotropy of the corresponding tensors;
typically, there is one large eigenvalue and the corresponding eigenvector yields the orientation of the fiber.  
In DTI, potential problems arise in areas where two or more fiber bundles are crossing because the tensors are not designed for the representation of multiple directions. In order to overcome this, the Q-ball imaging (QBI) approach \cite{tuch2004q,descoteaux2007regularized, hess2006q} uses higher angular information to allow for multiple directional peaks at each voxel; it has been applied to diffusion tractography \cite{behrens2007probabilistic}.
The Q-ball imaging data can be modeled by a probability density on the $3D$-unit sphere called  orientation distribution function (ODF). The corresponding space of ODFs can be endowed with a Riemannian manifold structure \cite{goh2009nonparametric}.

In the context of DTI, Wang and Vemuri consider 
a Chan-Vese model for manifold-valued data (which is a variant of the Potts model for the case of two segments)
and a piecewise smooth analogue \cite{wang2005dti, wang2004affine}.
Their method is based on a level-set active-contour approach
which iteratively evolves the jump set
followed by an update of the mean values (or a smoothing step for the piecewise smooth analogue) 
on each of the two segments.
In order to reduce the computational load in their algorithms
(caused by Riemannian mean computations for a very large amount of points)
the authors resort to non-Riemannian distance measures in \cite{wang2005dti, wang2004affine}. 
Recently, a fast recursive strategy for computing the Riemannian mean has been proposed   
and applied to the piecewise constant Chan-Vese model in \cite{cheng2012efficient}.
Related segmentation methods are $K$-means clustering \cite{wiegell2003automatic},  geometric flows  
\cite{jonasson2005white} or  level set methods \cite{feddern2003level,zhukov2003level}.

In this work, we propose algorithms 
for Mumford-Shah and Potts regularization for Riemannian manifolds 
(which includes DTI with the Fisher-Rao metric) for both signals and images.
For manifold-valued data, 
the distance $d$ in \eqref{eq:mumfordShah} becomes the Riemannian distance
and the differential $D$ can be understood in the sense of 
metric differentials \cite{kirchheim1994rectifiable}.
For univariate Mumford-Shah and Potts problems, we derive solvers based on a combination of dynamic programming techniques developed in \cite{mumford1989optimal,chambolle1995image,winkler2002smoothers,friedrich2008complexity}
and proximal point splitting algorithms for manifold-valued data developed by the authors in 
\cite{weinmann2013total}. 
Our algorithms are applicable 
for manifolds whose Riemannian exponential mapping
and its inverse can be evaluated in reasonable time.
For Cartan-Hadamard manifolds (which includes the manifold in DTI) 
our algorithms compute global minimizers for all input data.
(We note that the univariate  problems are not NP hard.)
These results actually generalize to the more general class of Hadamard spaces.
For Mumford-Shah and Potts problems for manifold-valued images (where the problems become NP-hard), 
we propose a novel splitting approach. 
Starting from a finite difference 
discretization of \eqref{eq:mumfordShah} we use a penalty method to split the problems into computationally tractable subproblems. These subproblems are closely related to univariate Mumford-Shah and Potts problems and can also be solved using the methods we developed for these problems in this paper.
We note that our methods neither require a priori knowledge on the number of segments
nor a discretization of the manifold.
We demonstrate the capabilities of our methods by applying them to two medical imaging modalities:
DTI and Q-ball imaging.
For DTI, we first consider several synthetic examples corrupted by Rician noise and show
our algorithms potential for edge-preserving denoising. 
As specific medical imaging application, we obtain a segmentation of the corpus callosum 
for real human brain data. 
We conclude with experiments for  Q-ball imaging.

\subsection{Organization of the article}

Section \ref{sec:Algos1d} deals with algorithms for the 
univariate Potts and Mumford-Shah problems for manifold-valued data.
We start by presenting a dynamic programming approach 
for the univariate Potts and Mumford-Shah problem in Section~\ref{subsec:DynProgPotts}.
Then we use this approach to derive 
an algorithm for univariate Potts functionals for manifold-valued data in
Section~\ref{subsec:AlgPotts1d} and to derive
an algorithm for the univariate Mumford-Shah problem in Section~\ref{subsec:AlgMumf1d}.	
An analysis of the derived algorithms is given in 
Section~\ref{subsec:Analysis1d}.
In Section~\ref{sec:Algos4Images}, 
we derive algorithms for the 
Potts and Mumford-Shah problems for manifold-valued images.
We first deal with proper discretizations and then
propose a suitable splitting into subproblems that we solve
using similar techniques as in the univariate case.
We apply our algorithm to DTI data in Section~\ref{sec:DTI_manifold} 
and to Q-ball data in Section~\ref{sec:QBall_manifold}.

%% file: pottsManifold_Algos.tex

\section{Univariate Mumford-Shah and Potts functionals for manifold-valued data}
\label{sec:Algos1d}

In this section,
we present solvers for Mumford-Shah and Potts problems for univariate manifold-valued data.
These are not only important in their own right; variants of the derived solvers
are also used as a basic building block for the proposed algorithm for the multivariate problems. 

We first deal with some general issues; then, we derive the announced algorithms -- 
first for the univariate Potts problem and then for the univariate Mumford-Shah problem; 
we conclude with an analysis of both algorithms.      

In the univariate case, the discretization of the Mumford-Shah functional \eqref{eq:mumfordShah}
and the Potts functional ($\alpha=\infty$ in \eqref{eq:mumfordShah}) is straightforward.
The (equidistantly sampled) discrete Mumford-Shah functional reads 
\begin{equation} \label{eq:1dMS_mani_jumpFormulation}
	B_{\alpha,\gamma}(x) = \frac1p \sum_{i = 1}^n d(x_i,f_i)^p + \frac{\alpha}{q}  \sum_{i \notin \mathcal{J}(x)} d(x_i,x_{i+1})^q  + \gamma |\mathcal{J}(x)|,
\end{equation}
where $d$ is the distance with respect to the Riemannian metric in the manifold $M,$ $f \in M^n$ is the data, and $\mathcal{J}$ is the jump set of $x.$ 
The jump set is given by $\mathcal{J}(x) = \{i: 1 \leq i < n \mbox{ and } d(x_i,x_{i+1}) > s\}$ 
where the jump height $s$ is related to the parameter $\gamma$ via $\gamma = \alpha s^q/q$.
Using a truncated power function we may rewrite \eqref{eq:1dMS_mani_jumpFormulation} in the Blake-Zisserman type form 
\begin{equation} \label{eq:1dMS_manifold_truncatedFormulation}
B_{\alpha,s}(x) = \frac1p \sum_{i = 1}^n d(x_i,f_i)^p + \frac{\alpha}{q}  \sum_{i=1}^{n-1} 
\min(s^q,d(x_i,x_{i+1})^q),
\end{equation}
where $s$ is the argument the power function $t \mapsto t^q$ is truncated at.

The discrete univariate Potts functional for manifold-valued data reads
\begin{equation} \label{eq:1dPotts_mani}
P_{\gamma}(x) = \frac1p \sum_{i = 1}^n d(x_i,f_i)^p + \gamma |\mathcal{J}(x)|,
\end{equation}
where $d$ is the distance in the manifold and $i$ belongs to the jump set of $x$
if $x_i \neq x_{i+1}.$

We first of all show that the problems \eqref{eq:1dMS_mani_jumpFormulation} and \eqref{eq:1dPotts_mani} have a minimizer. (We recall that certain variants of the continuous Mumford-Shah and Potts functional do not have a minimizer without additional assumptions;
see, e.g., \cite{fornasier2013existence}.)
\begin{theorem}\label{thm:ExMinim1D}
In a complete Riemannian manifold 
the discrete Mumford-Shah functional \eqref{eq:1dMS_mani_jumpFormulation} 
and the discrete Potts functional \eqref{eq:1dPotts_mani} have a minimizer.
\end{theorem}
The proof is given in Appendix~\ref{secAppendix:ExMinimizers}.
We note that the data spaces in applications 
are typically complete Riemannian manifolds.

\subsection{The basic dynamic program for univariate Mumford-Shah and Potts problems}
\label{subsec:DynProgPotts}

In order to find a minimizer of the Mumford-Shah problem  \eqref{eq:1dMS_mani_jumpFormulation}
and the Potts problem  \eqref{eq:1dPotts_mani}, we use a general dynamic programming principle which was considered for the corresponding scalar and vectorial problems  in various contexts; 
see, e.g., \cite{mumford1989optimal,chambolle1995image,winkler2002smoothers,friedrich2008complexity,weinmann2014l1potts,
storath2014jump}. 
We briefly recall the basic idea starting with the Mumford-Shah problem.
It is convenient to use the notation
\[
	x_{l:r} = (x_l,..., x_r).
\]
Assume that we have already computed minimizers $x^l$ of 
the functional $B_{\alpha,\gamma}$ associated with the partial data 
$f_{1:l} = (f_1,..., f_l)$ for each $l = 1,..., r-1$ and some $r \leq n.$ 
Then we compute $x^r$ associated to data $f_{1:r}$
as follows.
With each $x^{l-1}$ of length $l-1,$ 
we associate a candidate of the form $x^{l,r}=(x^{l-1},h^{l, r}) \in M^r$
which is the concatenation of $x^{l-1}$	
with a vector $h^{l, r}$ of length $r - l +1.$ 
This vector $h^{l, r}$ is a minimizer of the 
 problem
\begin{equation}\label{eq:epsilonGeneral}
\epsilon_{l,r} = \min_{h \in M^{r- l + 1}} \sum_{i=l}^{r-1} \frac\alpha p d^q(h_i, h_{i+1}) + \frac1p \sum_{i=l}^{r}
d^p(h_i, f_i),
\end{equation}
and $\epsilon_{l,r}$ is the error of a best approximation on the (discrete) interval $(l,...,r).$
Then we calculate the
quantity 
\begin{align}\label{eq:potts_value_candidate}
\min_{l=1,...,r} \{B_{\alpha,\gamma}(x^{l-1}) +  \gamma + \epsilon_{l,r} \},
\end{align}
which we will see to coincide with  
the minimal functional value  of $B_{\alpha,\gamma}$ for data $f_{1:r}$
(cf.~Theorem~\ref{thm:AlgProducesMinimizers} and Theorem~\ref{thm:AlgProducesMinimizersAdmissPartitions}).
Then, we set $x^{r} = x^{l^*, r},$ 
where $l^*$ is a minimizing argument in \eqref{eq:potts_value_candidate}.
We successively compute $x^r$ for each $r = 1,..., n$ until we end up with full data $f.$
Actually, only the $l^*$ and the $\epsilon_{l,r}$ and not the 
vectors $x^{r}$ have to be computed in this selection process; in a postprocessing step, 
the solution can be reconstructed from this information; see Algorithm~\ref{alg:manifold_1D} and  \cite{friedrich2008complexity} for further details.
With these improvements, the dynamic programming skeleton
(without the cost for computing the approximation errors $\epsilon_{l,r}$) 
has quadratic cost with respect to time and linear cost with respect to space.	
In practice, the computation can be accelerated significantly
by pruning the search space \cite{killick2012optimal,storath2014fast}.

In order to adapt the dynamic program for
 the Potts problem  \eqref{eq:1dPotts_mani}
 the only modification required is that 
 the approximation errors on the intervals $\epsilon_{l,r}$ read
\begin{equation}\label{eq:epsilonPotts}
	\epsilon_{l,r} = \min_{h \in M} \frac1p \sum_{i=l}^{r}	d^p(h, f_i),
\end{equation}
and the candidates are of the form
$x^{l,r} = (x^{l-1}, h^{l, r}),$
where $h^{l, r} \in M^{r-l+1}$ is constant and componentwise equals a minimizer $h^\ast$ of $\eqref{eq:epsilonPotts}$ on the interval $l,\ldots, r.$
We next deal with the computation of these minimizers.

\subsection{An algorithm for univariate Potts functionals for mani\-fold-valued data}
\label{subsec:AlgPotts1d}

In order to make the dynamic program from Section \eqref{subsec:DynProgPotts} work for 
the Potts problem for manifold-valued data, we see from Section \eqref{subsec:DynProgPotts}
that we have to compute the approximation errors $\epsilon_{l,r}$ given in \eqref{eq:epsilonPotts}
in the Riemannian manifold $M.$ This means we are faced with the problem 
of computing a minimizer for the manifold-valued data 
$f_{l:r}=(f_l,\ldots,f_{r})$ and then to calculate the corresponding approximation error.

We first consider the case $p=2$ which amounts to the ``mean-variance'' situation.
Since our data live in a Riemannian manifold, the usual vector space operations to define the arithmetic mean are not available. 
However, it is well known 
(cf. \cite{karcher1977riemannian,kendall1990probability, pennec2006riemannian, fletcher2007riemannian})
that a minimizer
\begin{equation} \label{eq:DefMean}
z^\ast \in \argmin_{z \in M} \sum_{i=1}^N d(z,z_i)^2
\end{equation}
is the appropriate definition of a mean $z^\ast \in \mean(z_1,\ldots,z_N)$  of the $N$ elements $z_i$ on the manifold $M$. 
A mean is in general not uniquely defined since the minimization problem has no unique solution in general.
If the $z_i$  are contained in a sufficiently small ball, 
however, the solution is unique. We then replace the ``$\in$'' symbol by an ``$=$'' symbol
and call $z^*$ the mean. 
The actual size of the ball where minimizers are unique depends on the sectional curvature of the manifold $M;$  
for details and for further information we refer to \cite{kendall1990probability, karcher1977riemannian}.  

In contrast to the Euclidean case there is no closed form expression of the intrinsic mean defined by \eqref{eq:DefMean} in Riemannian manifolds.
A widespread method for computing the intrinsic mean is the gradient descent approach
(already mentioned in \cite{karcher1977riemannian})
given by
\begin{equation} \label{eq:GradientDescentIntMean}
z^{(k+1)} =  \exp_{z^{(k)}} \sum_{i=1}^N \tfrac{1}{N} \exp^{-1}_{z^{(k)}}z_i.
\end{equation}
(Recall that the points $z_1,\ldots,z_N$ are the points for which the intrinsic mean is computed.)
Information on convergence related and other issues can, e.g., be found in the papers \cite{fletcher2007riemannian,afsari2013convergence} and the references therin.
Newton's method was also applied to this problem in the literature; see, e.g., \cite{ferreira2013newton}. 
It is reported in the literature and also confirmed by the authors' experience that the gradient descent converges rather fast; 
in most cases, $5$-$10$ iterations are enough.
This might explain why this relatively simple method is widely used.

For general $p\neq 1,$ the gradient descent approach works as well.
The case $p=1$ amounts to considering the intrinsic median and the intrinsic absolute deviation.
In this case, the gradient descent \eqref{eq:GradientDescentIntMean} is replaced by a subgradient descent
which in the differentiable part amounts to rescaling the tangent vector given on the right-hand side of \eqref{eq:GradientDescentIntMean} 
to length $1$ and considering variable step sizes which are square-integrable but not integrable; 
see, e.g., \cite{arnaudon2013approximating}.

A speedup using the structure of the dynamic program is obtained by 
initializing with previous output. More precisely, when starting the iteration of the mean for data 
$f_{l+1:r},$ we can use the already computed mean for the data $f_{l:r}$
as an initial guess. We notice that this guess typically becomes even better 
the more data items we have to compute the mean for, i.e., the bigger $r-l$ is. This is important since this case is the 
computational more expensive part and a good initial guess reduces the number of iterations needed. 
 
 A possible way to reduce the computation time further is to approximate the mean by a certain iterated two-point averaging 
construction (known as geodesic analogues in the subdivision context) 
as explained in \cite{wallner2005convergence}. 
Alternatively, one could use a ``$\log-\exp$'' construction (also known from subdivision; see \cite{rahman2005multiscale})
which amounts to stopping the iteration \eqref{eq:GradientDescentIntMean} after one step.
 
The proposed algorithm for univariate Potts functionals for manifold-valued data is summarized in Algorithm~\ref{alg:manifold_1D}.

 \begin{figure}
 	\def\subfigwidth{1\columnwidth}
 	\def\figurewidth{1.02\textwidth}
 	\def\hs{\hspace{-0.02\textwidth}}
 	\def\vs{\vspace{0.05\textwidth}}
 	\centering
 	\hs
 	\begin{subfigure}[t]{\subfigwidth}
 		\centering
 		\includegraphics[height=\figurewidth, angle=90]{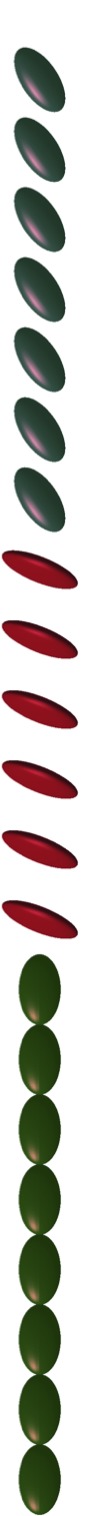}
 		\caption{}
 	\end{subfigure}\\
 	\vs
 	\hs 
 	\begin{subfigure}[t]{\subfigwidth}
 		\centering
 		\includegraphics[height=\figurewidth, angle=90]{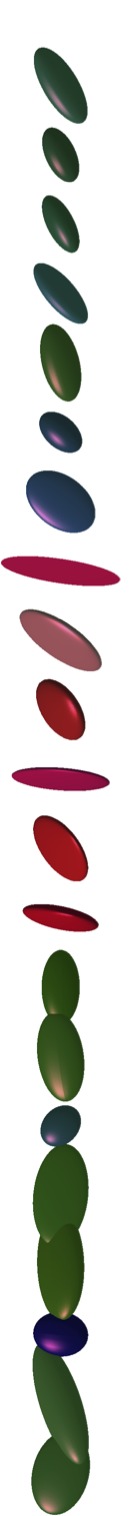}
 		\caption{}
 	\end{subfigure}\\
 	\vs
 	\hs
 	\begin{subfigure}[t]{\subfigwidth}
 		\centering
 		\includegraphics[height=\figurewidth, angle=90]{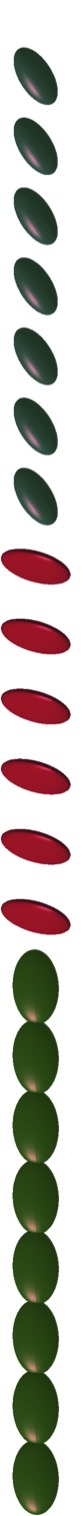}
 		\caption{}
 	\end{subfigure}
  	\caption{(a) Synthetic piecewise constant signal;
 		(b) noisy data (Rician noise with $\kappa = 85$); 
 		(c) Potts regularization ($p,q = 1$) using Algorithm~\ref{alg:manifold_1D} 
 		with parameter $\gamma = 84.5$. The signal is reconstructed almost perfectly;
 		the exact jump locations are obtained. 
 	}
 	\label{fig:ex2_1D}  
 \end{figure}

\subsection{An algorithm for univariate Mumford-Shah functio\-nals for manifold-valued data}
\label{subsec:AlgMumf1d}

In order to make the dynamic program from Section \ref{subsec:DynProgPotts} work for 
the Mumford-Shah problem with manifold-valued data, we have to compute the approximation errors $\epsilon_{l,r}$
in \eqref{eq:epsilonGeneral}. To this end, we compute minimizers of the problem
\begin{equation} \label{eq:minimization_problem}
V_\alpha(x;f) = \frac 1 p \sum_{i} d^p(x_{i},f_{i}) + \alpha \frac 1 q \sum_{i} d^q(x_{i},x_{i+1}). 
\end{equation}
Here $x$ is the target variable and $f$ is the data.
These are $L^p$-$V^q$ type problems: the data term is a manifold $\ell^p$ distance and the second term is a 
$q$th variation; in particular, $q=1$ corresponds to manifold-valued total variation. Solvers
for these problems have been developed in the authors' paper \cite{weinmann2013total}. 
We briefly recall the approach concentrating on the univariate case; for details
we refer to \cite{weinmann2013total}. 
We decompose the functional~\eqref{eq:minimization_problem}  into the sum $V_\alpha = F + \alpha \sum_i G_i,$ 
where we let
$G_i(x)$ $= \frac1q d^q(x_i,x_{i+1})$ 
and  $F(x) =  \frac1p \sum_i d^p(x_i,f_i).$
For each of these summands, we can explicitly compute their proximal mappings defined by
\begin{equation}\label{eq:DefProxy}
\prox_{\lambda G_{i}} x = \argmin_y  \left( \lambda G_{i}(y) +  \frac12  d^2(x,y) \right).
\end{equation}
They are given in terms of points on certain geodesics. In detail, we get
\begin{equation}\label{eq:ProxOfGi}
\begin{split}
(\prox_{\lambda G_{i}} x)_{i} &=  [x_{i},x_{i+1}]_t,  \\
(\prox_{\lambda G_{i}} x)_{i+1} &=  [x_{i+1},x_i]_t.
\end{split}
\end{equation}
where $[x,y]_t$ denotes the point reached after time $t$ on the unit speed geodesic which is starting in $x$ and going to $y$.
For the practically relevant cases $q=1,2$ the parameter $t$ has an explicit representation: 
for $q=1,$ we have $t = \lambda,$ if  $\lambda < \tfrac{1}{2} d(x_i,x_{i+1}),$ and $d(x_i,x_{i+1})/2$ else;
for $q=2$ we get $t =  \frac{\lambda}{1+ 2 \lambda} d(x_{i},x_{i+1}).$
Similarly, the proximal mapping of $F$ is given by  
\begin{align} \label{eq:ProxOfF}
(\prox_{\lambda F})_{i}(x) =  [x_{i},f_{i}]_s.    
\end{align}
For $p=1,$ we have $s = \lambda$ if $\lambda < d(x_{i},f_{i}),$  and
$d(x_{i},f_{i})$ else; for $p=2,$ we obtain that $s= \tfrac{\lambda}{1+\lambda} d(x_{i},f_{i})$.  
We notice that the above proximal operators are uniquely defined if there is precisely one shortest geodesic joining the two points involved. 
Otherwise, one has to resort to set-valued mappings. Uniqueness is given for the class of Cartan-Hadamard manifolds
which includes the data space in DTI considered in Section~\ref{sec:DTI_manifold}. 

Equipped with these proximal mappings we apply a cyclic proximal point algorithm for manifold-valued data \cite{bavcak2013computing}: 
we apply the proximal mappings of $F,\alpha G_r,\ldots, \alpha G_l$ (with parameter $\lambda$) and iterate this procedure.
During the iteration, we decrease the parameter $\lambda_k$ in the $k$th iteration in a way such that
$\sum_k \lambda_k = \infty$ and $\sum_k \lambda_k^2 < \infty.$ 

A speedup using the structure of the dynamic program is obtained by 
initializing with previous output as explained for the Potts problem in Section~\ref{subsec:AlgPotts1d}.  
The proposed algorithm for univariate Mumford-Shah functionals with manifold-valued data is summarized in Algorithm~\ref{alg:manifold_1D}.

\begin{algorithm}[ht]
\SetCommentSty{itshape}
	\Begin{
		\tcp{Find optimal partition}
				$B_0 \leftarrow  -\gamma$\;
		\For{$r\leftarrow 1,...,n$}{
			\For{$l\leftarrow 1,...,r$}{
				\tcp{Mumford-Shah case (Sec.~\ref{subsec:AlgMumf1d}):}
			$\epsilon \leftarrow \min_{h \in M^{r-l+1}} V_\alpha(h;f_{l:r})$ \tcp{use Alg. of Sec. \ref{subsec:AlgMumf1d}} 
				\tcp{Potts case (Sec.~\ref{subsec:AlgPotts1d}):}
			$\epsilon \leftarrow \min_{h \in M} \sum_{i=l}^r d^p(h,f_i)$ \tcp{use Alg. of Sec. \ref{subsec:AlgPotts1d}}
				$b \leftarrow B_{l-1} + \gamma + \epsilon;$\\
				\If{$b<B_r$}
				{
					$B_r \leftarrow b;$\\
					$p_r \leftarrow l-1;$
				}
			}
			}
			\tcp{Reconstruct solution from partition}
			$r\leftarrow n; l \leftarrow p_r$;\\
			\While{ $l>0$ }{
								\tcp{Mumford-Shah case (Sec.~\ref{subsec:AlgMumf1d}):}
				$h^* \leftarrow \argmin_{h \in M^{r-l+1}} V_\alpha(h;f_{l+1:r})$ \tcp{use Alg. of Sec. \ref{subsec:AlgMumf1d}}
				\tcp{Potts case (Sec.~\ref{subsec:AlgPotts1d}):}
				$h' \leftarrow \argmin_{h \in M} \sum_{i=l+1}^r d^p(h,f_{i})$
				\tcp{use Alg. of Sec. \ref{subsec:AlgPotts1d}}
				$h^* \leftarrow (h',\ldots,h')$\;
				$x^*_{l+1:r} \leftarrow h^*$\;
				$r\leftarrow l; l \leftarrow p_r$;\\
			}
			\Return{$x^*$}
	}
	\caption{Algorithm for the Mumford-Shah problem \eqref{eq:1dMS_mani_jumpFormulation} and the Potts problem \eqref{eq:1dPotts_mani} for univariate manifold-valued data}
	\label{alg:manifold_1D}
\end{algorithm}

\begin{figure}
	\def\subfigwidth{1\columnwidth}
	\def\figurewidth{1.02\textwidth}
	\def\vs{\vspace{0.04\textwidth}}
	\centering
	\begin{subfigure}[t]{\subfigwidth}
		\centering
		\includegraphics[height=\figurewidth, angle=-90]{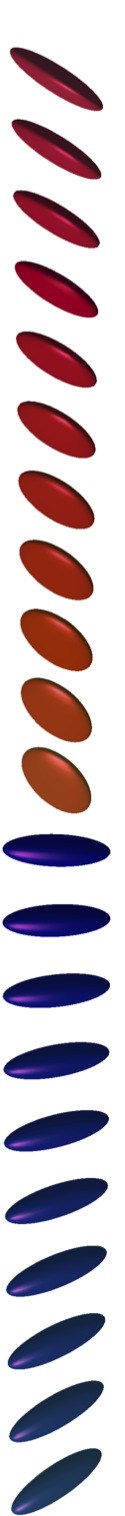}
		\caption{}
	\end{subfigure}\\
	\vs
	\begin{subfigure}[t]{\subfigwidth}
		\centering
		\includegraphics[height=\figurewidth, angle=-90]{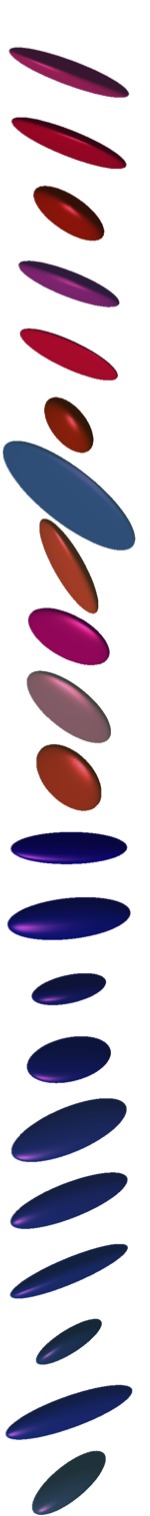}
		\caption{}
	\end{subfigure}\\
	\vs 
	\begin{subfigure}[t]{\subfigwidth}
		\centering
		\includegraphics[height=\figurewidth, angle=-90]{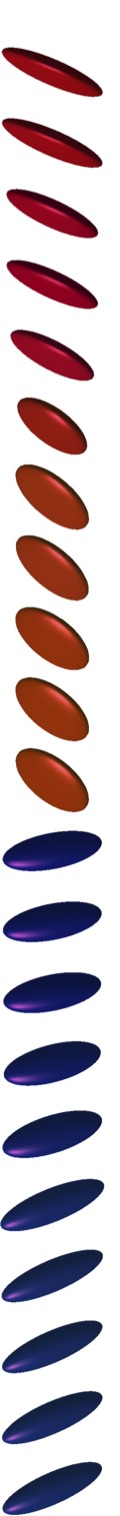}
		\caption{}
	\end{subfigure}
	\caption{(a) Synthetic piecewise smooth signal;
		(b) noisy data (Rician noise with $\kappa = 70$); 
		(c) Mumford-Shah regularization ($p,q = 1$) using Algorithm~\ref{alg:manifold_1D}
		with parameters $\alpha = 1.45$ and $\gamma = 1.5$.
		The noise is removed while preserving the jump.		
	}
	\label{fig:ex3_1D}  
\end{figure}

\subsection{Analysis of the univariate Potts and Mumford-Shah algorithms }
\label{subsec:Analysis1d}

We first obtain that our algorithms yield global minimizers for data in the class of Cartan-Hadamard manifolds 
which includes many symmetric spaces.
Prominent examples are the spaces of positive matrices (which are the data space in diffusion tensor imaging) and the hyperbolic spaces.
These are complete simply-connected Riemannian manifolds of nonpositive sectional curvature.  
For details we refer to \cite{doCarmo1992ri} or to \cite{ballmann1985manifolds}.
In particular, in these manifolds, geodesics always exist and are unique shortest paths.

\begin{theorem}\label{thm:AlgProducesMinimizers}
In a Cartan-Hadamard manifold,
Algorithm~\ref{alg:manifold_1D} 
produces a global minimizer for the univariate Mumford-Shah problem \eqref{eq:1dMS_mani_jumpFormulation}
(and the discrete Potts problem \eqref{eq:1dPotts_mani}, accordingly).   
\end{theorem}

The proof is given in  Appendix \ref{secAppend:UnivarMinimAlgo}.

We notice that this result generalizes to the more general class of (locally compact) Hadamard spaces. 
These are certain metric spaces generalizing the concept of Cartan-Hadamard manifolds; see, e.g., \cite{sturm2003he}.
Examples of Hadamard spaces which are not Cartan-Hadamard manifolds are the metric trees in \cite{sturm2003he}.
The validity of Theorem \ref{thm:AlgProducesMinimizers} for (locally compact) Hadamard spaces may be seen by inspecting
the proof noticing that all steps rely only on features of these spaces.

For analysis of general complete Riemannian manifolds, 
we first notice that, in this case we have to deal with questions of well-definedness.
We consider the Potts functional and data $f_1,\ldots,f_n.$ 
For each (discrete) subinterval $[l,r],$ a corresponding mean $h^{l,r}$ is defined as a minimizer of \eqref{eq:DefMean}
for data $f_l,\ldots,f_r.$ 
Although such a minimizer exists by the coercivity and continuity of the functional,
it might not be unique. Furthermore,
an algorithm such as gradient descent only computes a local minimizer for general input data. For data not too far apart, however, 
the gradient descent produces a global minimizer of \eqref{eq:DefMean} (since then the corresponding functional is convex).
If data are so far apart that the operations in the manifold are not even well-defined it might be likely that they do not belong to the same segment. Hence, let us consider a constant $C_K$ such that, if points belong to a $C_K$-ball with center in the compact set $K,$
then their mean is uniquely defined and obtained by converging gradient descent.
Assuming that  the data lie in $K,$ we call a partition of $[1,n]$ admissible if for any interval $[l,r]$ in this partition the corresponding data $f_{l:r}$ are centered in a common $C_K$-ball.
We get the following result.

\begin{theorem}\label{thm:AlgProducesMinimizersAdmissPartitions}
	Let $M$ be a complete Riemannian manifold. Then the  univariate Potts problem given in Algorithm \ref{alg:manifold_1D} with $p=2$
	produces a minimizer of the discrete Potts problem \eqref{eq:1dPotts_mani} when restricting the search space  to candidates whose jump sets correspond to admissible partitions.  
\end{theorem}

The proof can be found in Appendix \ref{secAppend:UnivarMinimAlgo}. This result can be easily generalized to the general case $p\geq 1.$

\section{Mumford-Shah and Potts problems for mani\-fold-valued images}
\label{sec:Algos4Images}

	We now consider Mumford-Shah and Potts regularization 
	for manifold-valued images.
	In contrast to the univariate case,
	finding global minimizers is not tractable anymore in general.
	In fact, the Mumford-Shah problem and the Potts problem are known to be NP hard 
	in dimensions higher than one even for scalar data \cite{veksler1999efficient,
	alexeev2010complexity}.
	Therefore, the goal is to derive approximative strategies that 
	perform well in practice.
	
In the following it is convenient 
to use the notation $d^p(x,y)$ for the $p$-distance of two manifold-valued images $x,y \in M^{m\times n},$
i.e.
\[
d^p(x,y) = \sum_{i,j}d^p(x_{ij}, y_{ij}).
\]
We further define the penalty function
\[
\Psi_{a}(x) = \sum_{i,j} \psi(x_{(i,j) + a}, x_{ij})
\]
with respect to some finite difference vector $a \in \Z^2 \setminus \{0\}.$
Here, we instantiate the potential function $\psi$ 
in the  Mumford-Shah case by
\begin{equation}\label{eq:psiMumfi}
\psi(w,z) =  \frac{1}{q}\min(s^q, d(w, z)^q).
\end{equation}
and in the Potts case by
\begin{equation}\label{eq:psiPotts}
\psi(w,z) = 
\begin{cases}
	1, &\text{if }w \neq z,\\
	0, &\text{if }w = z,\\
\end{cases}
\end{equation}
for $w, z \in M.$

In higher dimensions, 
the discretization of the Mumford-Shah and Potts problem
is not as straightforward as in the univariate case.
A simple finite difference discretization
with respect to the coordinate  directions
is known to produce undesired block 
artifacts in the reconstruction 
\cite{chambolle1999finite}.
The results improve significantly 
when including further finite differences 
such as the diagonal directions \cite{chambolle1999finite,storath2014fast, storath2014joint}.
We here use a discretization of the general form
	\begin{equation}\label{eq:msDiscrete}
\min_{x \in M^{m\times n}}  \frac{1}{p} d^p(x, f) + \alpha\sum_{s=1}^R \omega_s \Psi_{a_s}(x),
\end{equation}
where the finite difference vectors $a_s \in \Z^2\setminus\{0\}$ belong to a neighborhood system $\Nc.$
The values $\omega_1,..., \omega_R$ are
 non-negative weights.
 We  focus on the neighborhood system
 \[
 	\Nc = \{  (1,0); (0,1); (1,1); (1,-1) \}
 \]
with the weights $\omega_1 = \omega_2 = \sqrt{2} -1$ and $\omega_3 = \omega_4 = 1- \frac{\sqrt{2}}{2}$
as in \cite{storath2014fast}.
 For further neighborhood systems and weights we refer to \cite{chambolle1999finite,storath2014fast}.
We next show the existence of minimizers of the discrete functional \eqref{eq:msDiscrete}.

\begin{theorem}\label{thm:ExMinimizers2D}
	Let $M$ be a complete Riemannian manifold. Then the discrete Mumford-Shah and Potts problems \eqref{eq:msDiscrete} both have a minimizer.
\end{theorem}
The proof is given in Appendix~\ref{secAppendix:ExMinimizers}.

	We next propose a splitting approach for the discrete
	 Mumford-Shah and Potts problems.
	To this end, we rewrite \eqref{eq:msDiscrete}
	as the constrained problem
	\begin{equation}\label{eq:constrained}
	\begin{split}
	&\min_{x_1,..., x_R}~\sum_{s=1}^R  \frac{1}{pR} d^p(x_s, f) + \alpha\omega_s \Psi_{a_s}(x_s)  \\
	&\text{subject to }  x_s = x_{s+1}  \text{ for all }1 \leq s \leq R.
	\end{split}
	\end{equation}
	Here, we use the  convention $x_{R+1} = x_1.$
	(Note that  $x_1,..., x_R$ are ${m\times n}$ images.)
	We use a penalty method (see e.g.~\cite{bertsekas1976multiplier})
	to include the constraints into the target functional and get the problem
	\begin{equation*}
	\min_{x_1,..., x_R}~\sum_{s=1}^R \omega_s pR \alpha \Psi_{a_s}(x_s) +  d^p(x_s, f)
	+ \mu_k d^p(x_s, x_{s+1}).
		\end{equation*}
	We use an increasing coupling sequence $(\mu_k)_k$ which fulfills the summability condition
	 $\sum_k \mu_{k}^{-1/p} < \infty.$
	Optimization with respect to all variables simultaneously is still not tractable,
	but our specific
	 splitting allows us to minimize the functional blockwise,
	 that is, with respect variables $x_1,..., x_R$ separately. 
	Performing the blockwise minimization we get
	the algorithm
	\begin{equation}\label{eq:OurSplitting}
	\left\{
	\begin{split}
		&\begin{split}
	x_1^{k+1} \in \argmin_{x}pR\omega_1 \alpha \Psi_{a_1}(x) &+ d^p(x, f)  
	             + \mu_k d^p(x, x_{R}^k), 
	     \end{split} \\[3ex]
	   &\begin{split}
	x_2^{k+1} \in \argmin_{x}pR\omega_2 \alpha \Psi_{a_2}(x) &+  d^p(x, f) 
		+ \mu_k d^p(x, x_{1}^{k+1}), \\
		\end{split}\\
	&\quad\vdots \\[2ex]
	&\begin{split}
	x_R^{k+1} \in		\argmin_{x}pR\omega_R  \alpha\Psi_{a_R}(x) &+  d^p(x, f) 
	 + \mu_k d^p(x, x_{R-1}^{k+1}).
	 	\end{split}
	\end{split}
	\right.
	\end{equation}
	We notice that each line of \eqref{eq:OurSplitting}
	decomposes into univariate subproblems of Mumford-Shah and Potts 
	type, respectively.
For example, we obtain
	\begin{equation}
		(x_{1})_{:,j} \in  \argmin_{z \in M^{n}}~pR \omega_1 \alpha \Psi(z) +  d^p(z, f_{:,j}) 
	+ \mu_k d^p(z, (x_{R}^{k})_{:,j})
	\end{equation}
	for the direction $a_1 = (1,0).$

The subproblems    are almost identical 
 with the univariate problems of Section \ref{sec:Algos1d}.
 Therefore, we can use the algorithms developed in 
 Section \ref{sec:Algos1d} with the following minor modification.
 For the Potts problem, the approximation errors are now instantiated by
\begin{equation*}
	\epsilon_{l,r} =  \min_{h \in M} \sum_{i=l}^{r}	d^p(h, f_{ij}) + 
	\mu_k d^p(h, (x_R^k)_{ij}),
\end{equation*}
for the subproblems with respect to direction $a_1$
(and analogously for the other directions $a_2,...,a_R$.)
This quantity can be computed 
by the gradient descent explained in Section~\ref{subsec:AlgPotts1d}.
In the Mumford-Shah case,
we have
\begin{equation*}
	\epsilon_{l,r} =  \min_{h \in M^{r- l + 1}} \sum_{i=l}^{r-1} pR \omega_1 \alpha d^q(h_i, h_{i+1})  \\ + \sum_{i=l}^{r}
	d^p(h_i, f_{ij}) 
	+ \sum_{i=l}^{m}
	\mu_k d^p(h_i, (x_R^k)_{ij}).
\end{equation*}
The only difference to
\eqref{eq:epsilonGeneral}
is the  extra \enquote{data term} 
\[
F'(h) = \sum_{i=l}^{r} d^p(h_i, (x_R^k)_{ij}).
\]
Its proximal mapping has the same 
form as the proximal mapping of $F$ 
in Section~\ref{subsec:AlgMumf1d}.
Thus, we only need to complement  
the cyclic proximal point algorithm for the $L^p$-$V^q$ problem
of Section~\ref{subsec:AlgMumf1d} 
by an evaluation of the proximal mapping with respect to $F'.$

We eventually show convergence.
	\begin{theorem} \label{thm:Convergence2D}
		For Cartan-Hadamard manifold-valued images 
		the algorithm \eqref{eq:OurSplitting} for both the Mumford-Shah and the Potts problem converge.
	\end{theorem}
	The proof is given in Appendix~\ref{secAppend:Convergence2D}.

%% file: pottsManifold_DTI.tex

\begin{figure}
	\def\subfigwidth{0.23\textwidth}
	\def\figurewidth{1\textwidth}
	\def\hs{\hspace{0.05\textwidth}}
	\centering
	
	\begin{subfigure}[t]{\subfigwidth}
		\includegraphics[width=\figurewidth]{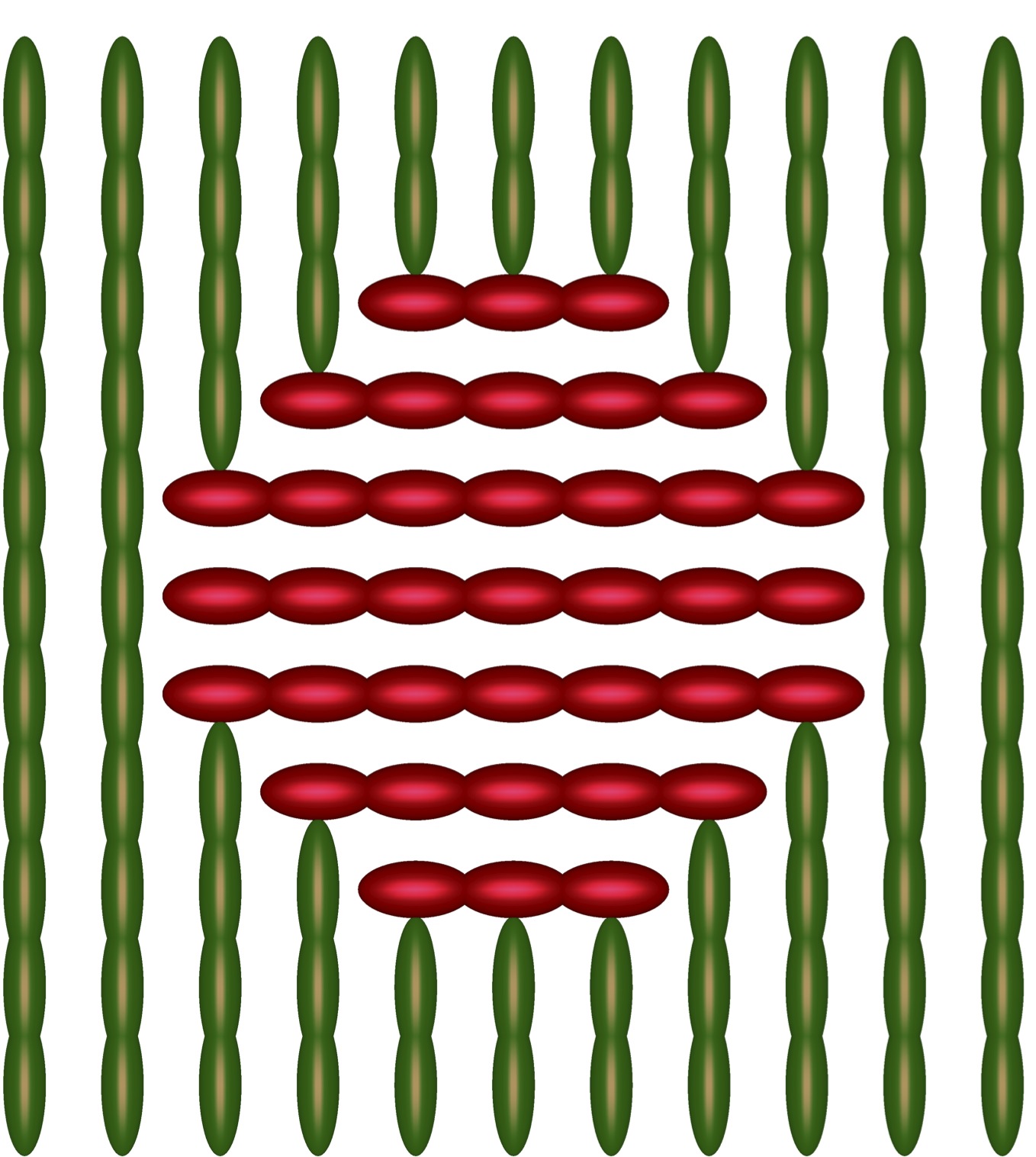}
		\caption{}
	\end{subfigure}
	\hs
	\begin{subfigure}[t]{\subfigwidth}
		\includegraphics[width=\figurewidth]{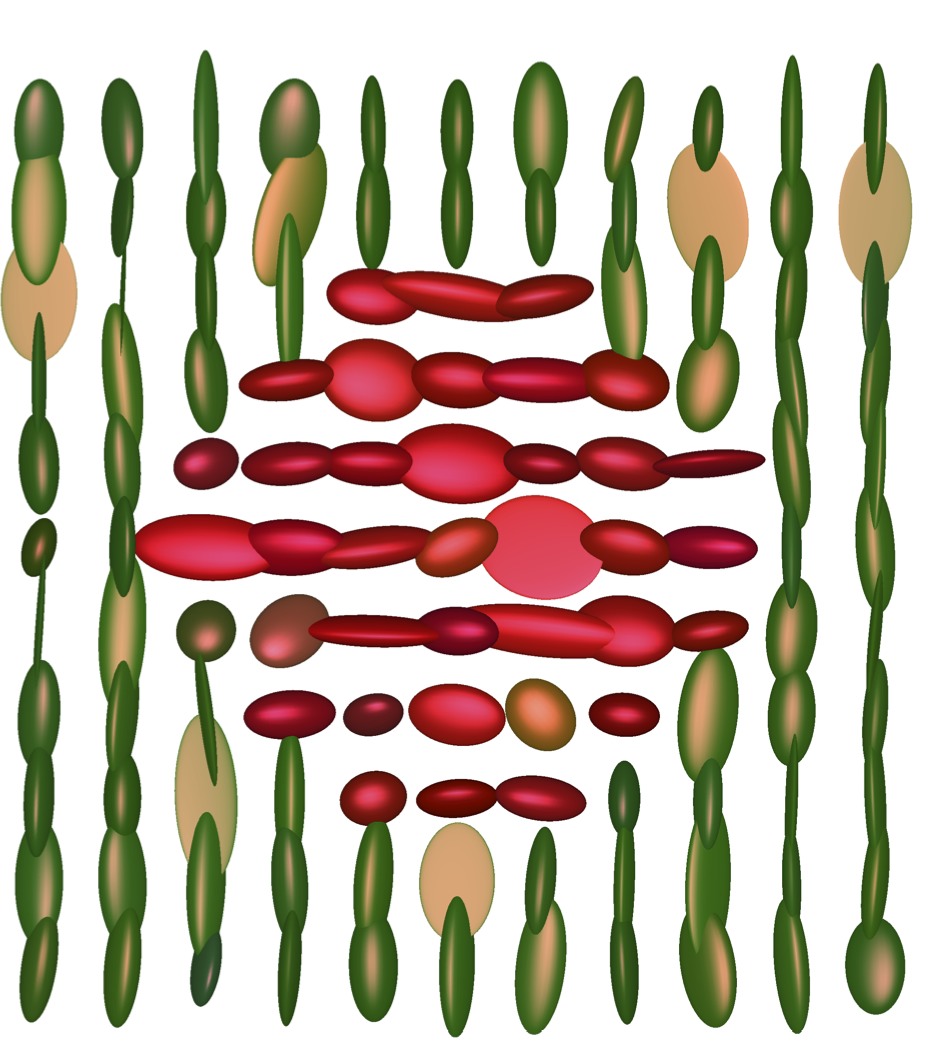}
		\caption{    }
	\end{subfigure}
	\\
	
	\begin{subfigure}[t]{\subfigwidth}
		\includegraphics[width=\figurewidth]{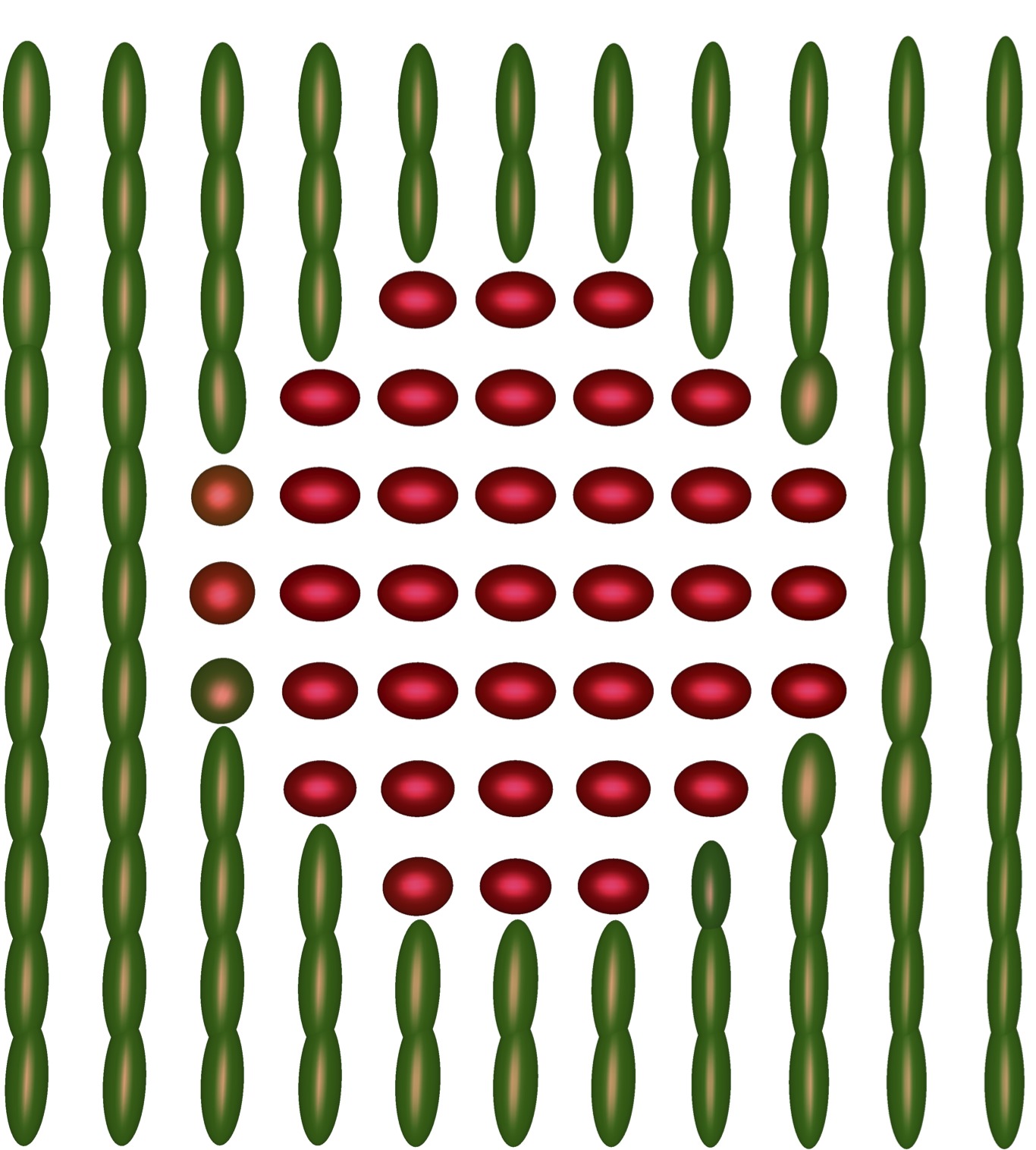}
		\caption{}
	\end{subfigure}
	\hs
	\begin{subfigure}[t]{\subfigwidth}
		\includegraphics[width=\figurewidth]{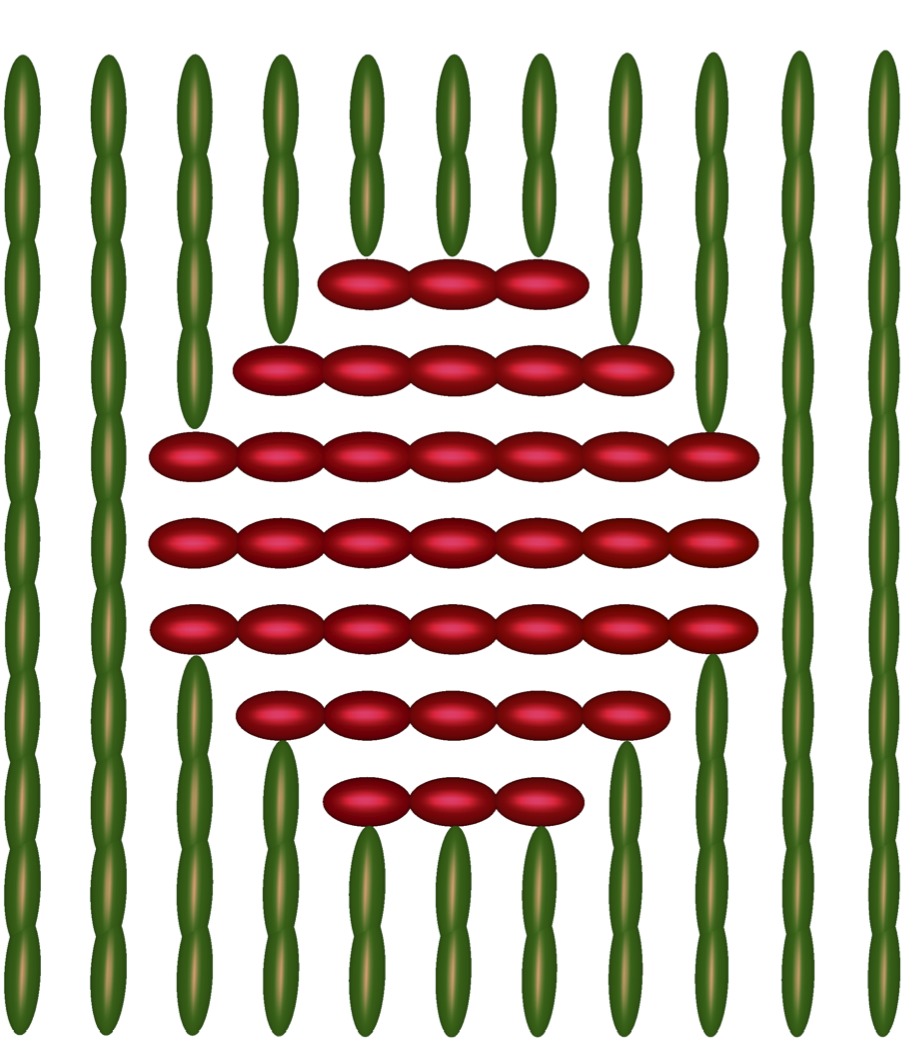}
		\caption{}
	\end{subfigure}
	\caption{(a) Synthetic DT image, 
		(b) noisy data (Rician noise of level $75$),
		(c) $L^1$-$TV$ reconstruction  (using TV parameter $\alpha = 0.65$), 
		(d) Potts reconstruction ($p,q=1$) with parameter $\gamma = 3.75$.
		While the $L^1$-$TV$ reconstruction decreases the contrast significantly, 
		the Potts method yields an almost perfect reconstruction.
	}
	\label{fig:synthDTI2D}
\end{figure}

\section{Application to Diffusion Tensor Images}
\label{sec:DTI_manifold}

The first application of our method is edge preserving denoising
and segmentation of diffusion tensor images.
Diffusion tensor imaging (DTI) is a non-invasive modality for medical imaging quantifying diffusional characteristics of a specimen.
It is based on nuclear magnetic resonance \cite{basser1994mr, johansen2009diffusion}.
Prominent applications are the determination of fiber tract orientations \cite{basser1994mr}, the detection of brain ischemia \cite{le2001diffusion}, and studies on autism \cite{alexander2007diffusion}, to mention only a few.
Regularization of DT images is important in its own right and, in particular, serves as a processing step in many applications. 
It has been studied in a  number of papers; we exemplarily mention
 \cite{wang2005dti, chen2005noise,  pennec2006riemannian, basu2006rician}.

In DTI, the diffusivity of water molecules is encoded into a so-called diffusion tensor.
This means that the data sitting in each pixel (or voxel) of a diffusion tensor image
is a positive (definite symmetric) $3 \times 3$ matrix $D.$
The space of positive matrices $\Pos$
is a Riemannian manifold when equipped with the Riemannian metric
\begin{equation}\label{eq:DTImetric}
g_D(W,V) = \trace(D^{-\tfrac{1}{2}} W D^{-1} V D^{-\tfrac{1}{2}});
\end{equation}
for details, see, e.g., \cite{pennec2006riemannian}.
Here the symmetric matrices $W,V$ represent tangent vectors in the point $D.$
Besides its mathematical properties,
the practical advantage of the Riemannian metric \eqref{eq:DTImetric}
in comparison to the Euclidean metric is that it reduces the swelling effect (\cite{tschumperle2001diffusion,arsigny2005fast}).
On the flipside, the algorithms and the corresponding theory become more involved.

\subsection{Implementation of our algorithms for DTI}

We now implement our algorithms for Mumford-Shah and Potts regularization for DTI data.
Due to the generality of our algorithms, we only need an implementation of the Riemannian exponential mapping
and its inverse to make them work on the concrete manifold.
For the space of positive matrices, the Riemannian exponential mapping $\exp_D$ is given by
\[
    \exp_D(W) = D^{\frac12} \exp(D^{-\frac12} W D^{-\frac12}) D^{\frac12}.
\]
Here $D$ is a positive matrix and the symmetric matrix $W$ represents a tangent vector in $D.$
The mapping $\exp$ is the matrix exponential.
The inverse of the Riemannian exponential mapping is given
by
\[
    \exp^{-1}_D(E) = D^{\frac12} \log(D^{-\frac12} E D^{-\frac12}) D^{\frac12}.
\]
for positive matrices $D,E.$
The matrix logarithm $\log$ is well-defined since the argument is a positive matrix.
The matrix exponential and logarithm can be efficiently computed by diagonalizing
the symmetric matrix under consideration and then applying the scalar exponential and logarithm functions
to the eigenvalues.
The distance between $D$ and $E$ is just the length of the tangent vector $\exp^{-1}_D(E)$ which can be explicitly
calculated by
$
    \dist(D, E) = (\sum_{l=1}^3 \log(\kappa_l)^2)^{\frac12},
$
where $\kappa_l$ is the $l^{\mathrm{th}}$ eigenvalue of the
matrix $D^{-\frac12} E D^{-\frac12}.$

The space of  positive matrices becomes a Cartan-Hadamard manifold with
the above Riemannian metric \eqref{eq:DTImetric}.
Hence the theory developed in this paper fully applies; in particular,
the univariate algorithms for DTI data produce global minimizers for all input data (see Theorem~\ref{thm:AlgProducesMinimizers}); 
furthermore, the algorithm \eqref{eq:OurSplitting} 
 converges, and all its subproblems are solved exactly.

\subsection{Synthetic data}

\begin{figure}[t]
         \def\subfigwidth{0.35\textwidth}
        \def\figurewidth{1\textwidth}
         \def\hs{\hspace{0.05\textwidth}}
    \centering
    \begin{subfigure}[t]{\subfigwidth}
    \includegraphics[width=\figurewidth]{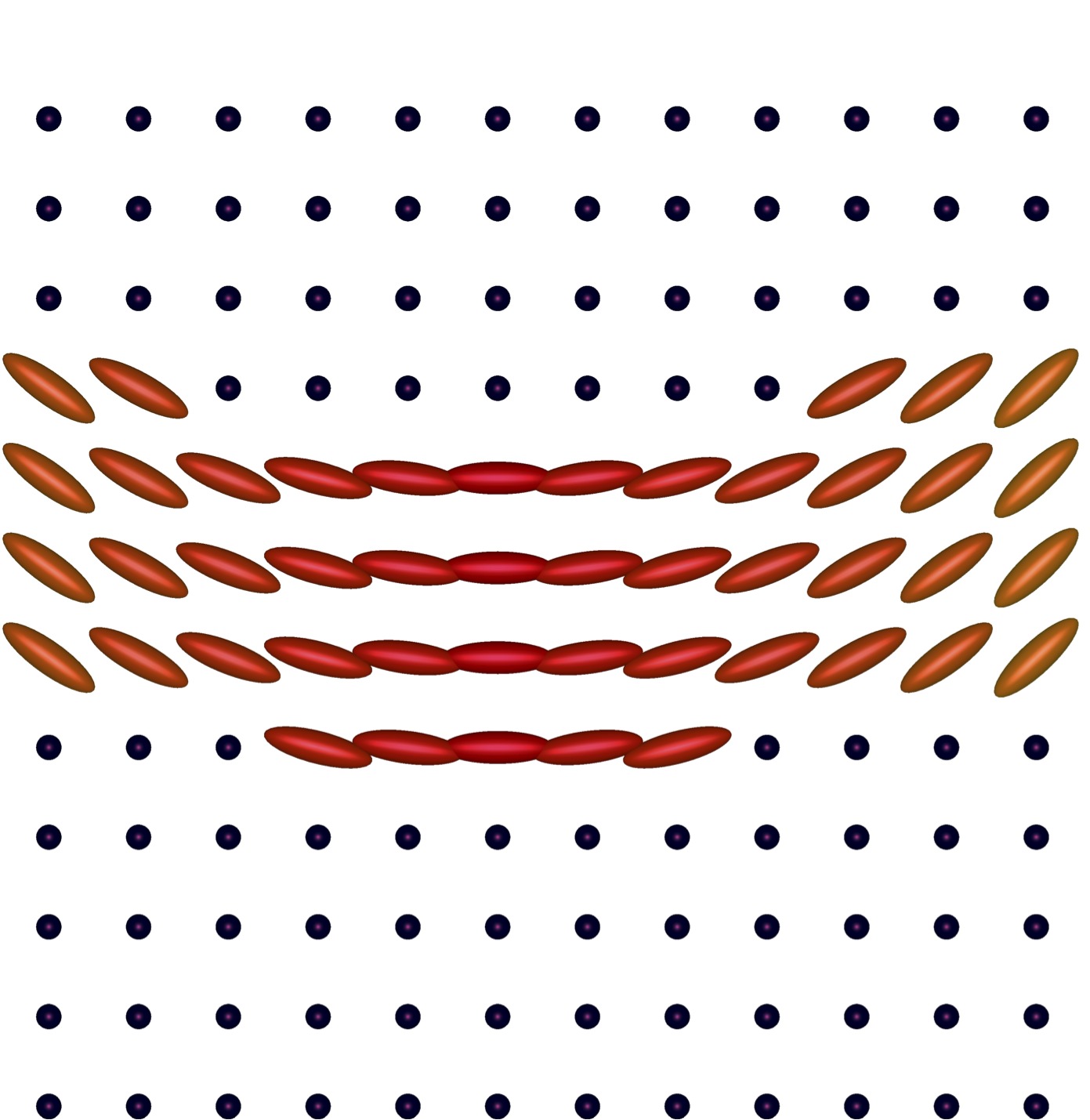}
    \caption{
    }
  \end{subfigure}
    \hs
    \begin{subfigure}[t]{\subfigwidth}
    \includegraphics[width=\figurewidth]{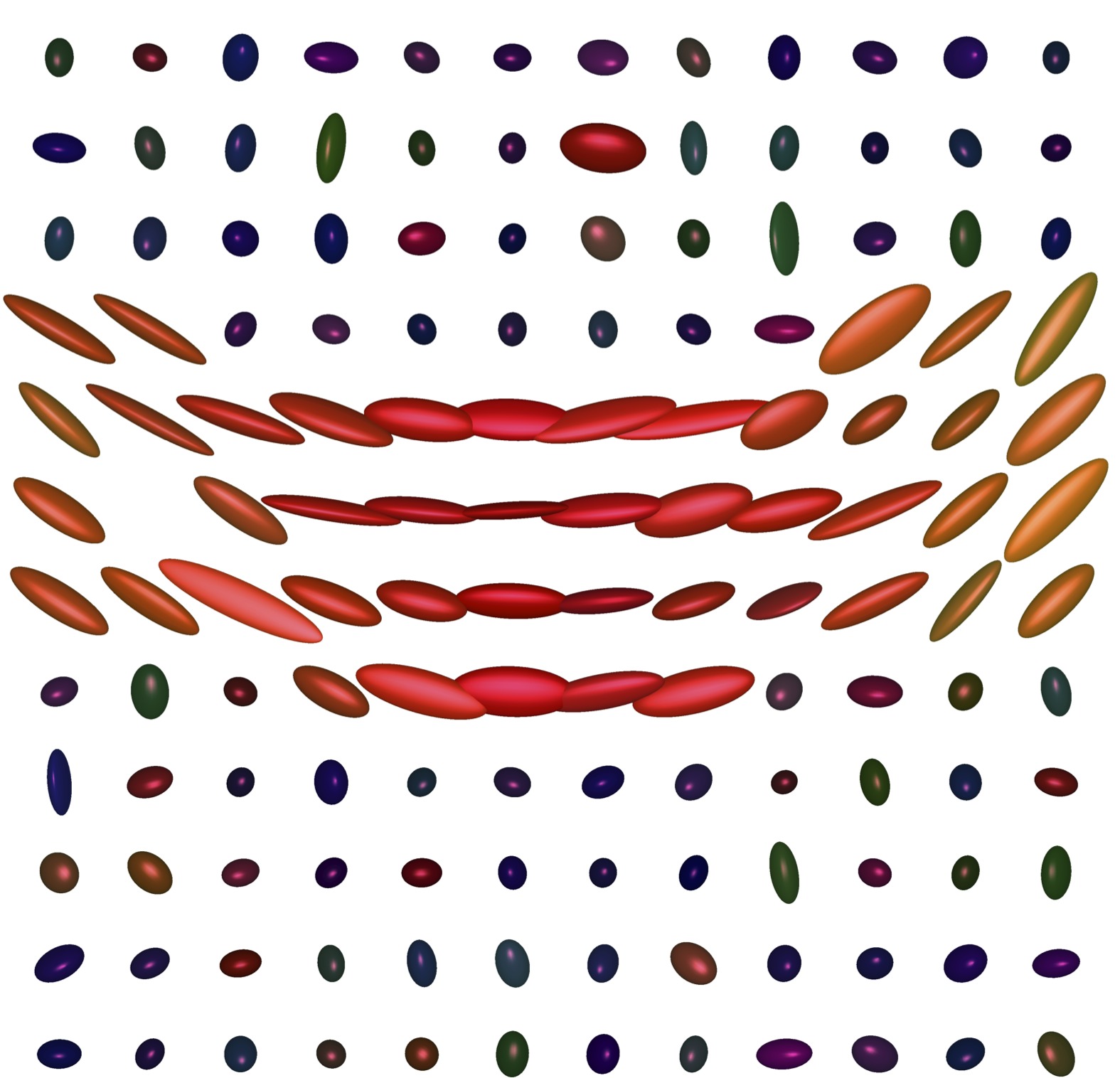}
    \caption{}
  \end{subfigure}
    \\
    \begin{subfigure}[t]{\subfigwidth}
    \includegraphics[width=\figurewidth]{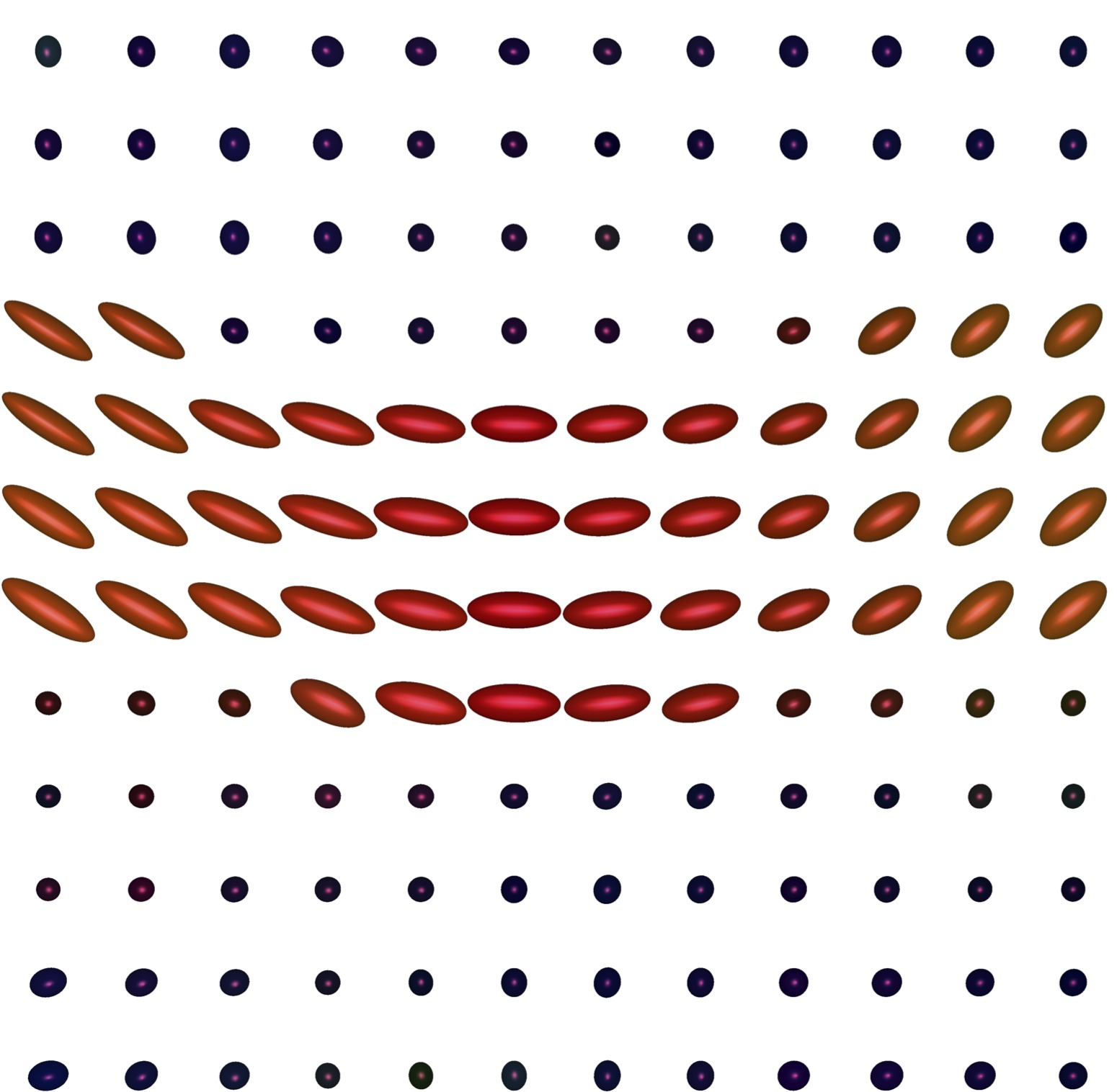}
    \caption{}
  \end{subfigure}
     \hs
    \begin{subfigure}[t]{\subfigwidth}
    \includegraphics[width=\figurewidth]{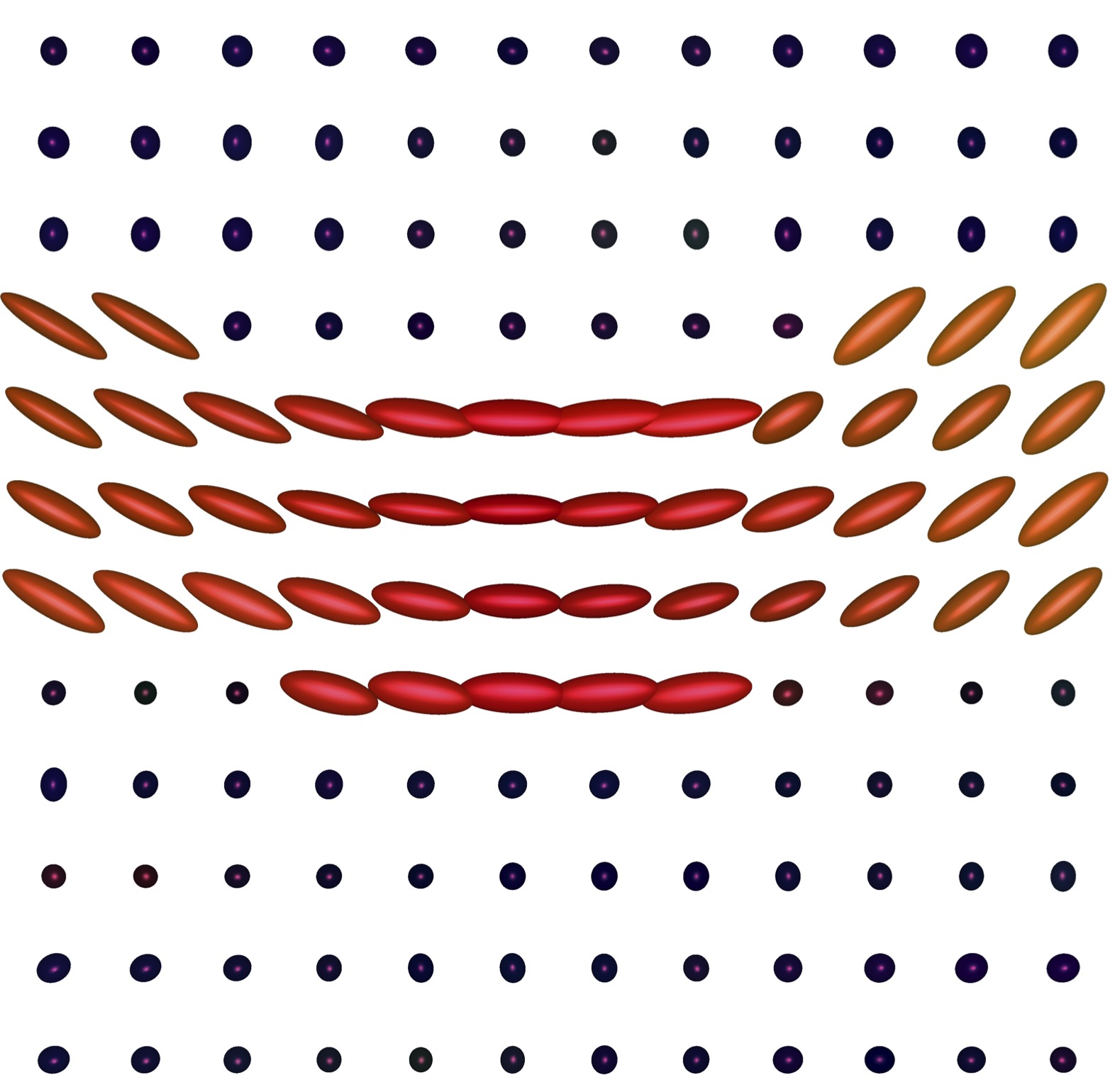}
    \caption{}
  \end{subfigure}
\caption
{(a) Synthetic DT image, 
 (b) noisy data (Rician noise of level $65$),
 (c) $L^1$-$TV$ reconstruction (TV parameter $\alpha = 0.5$), 
 (d) Mumford-Shah reconstruction ($p,q=1$) using parameters $\gamma = 0.8$ and $\alpha = 5$. 
    The noise is removed and the segments of the original image are recovered reliably.
}
\label{fig:synthCC}
\end{figure}

\begin{figure*}[t]
	 	\def\subfigwidth{0.48\textwidth}
		\def\figurewidth{0.65\textwidth}
	 	 \def\figureheight{0.7\textwidth}
		 \def\hs{\hspace{0.05\textwidth}}
	\centering
	\begin{subfigure}[t]{\subfigwidth}
    \includegraphics[width=1\textwidth, trim=0 0 0 45, clip]{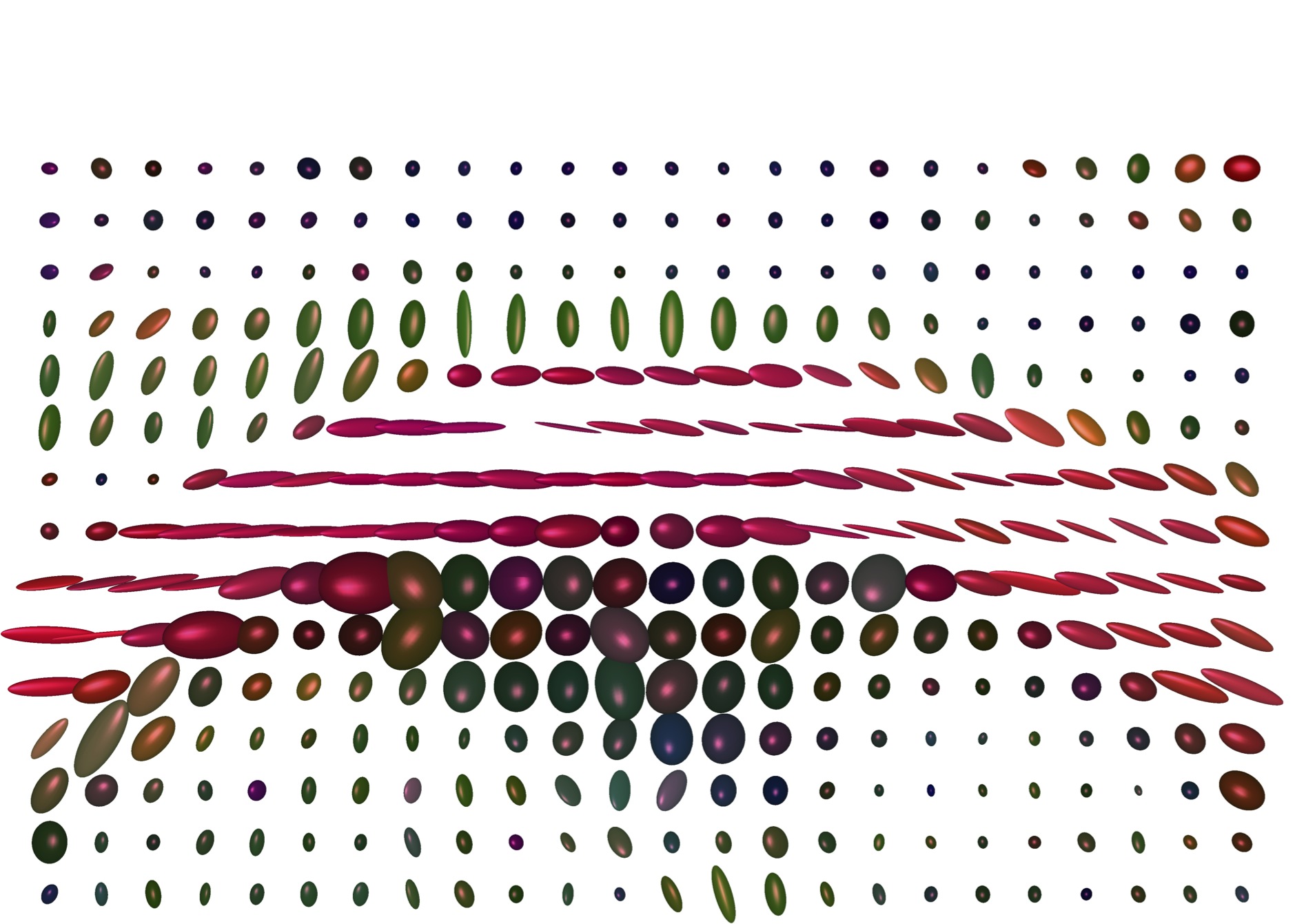} 
    \caption{}
  \end{subfigure}
  \hfill
	\begin{subfigure}[t]{\subfigwidth}
    \raisebox{0.8ex}{\includegraphics[width=0.98\textwidth, trim=0 0 0 45, clip]{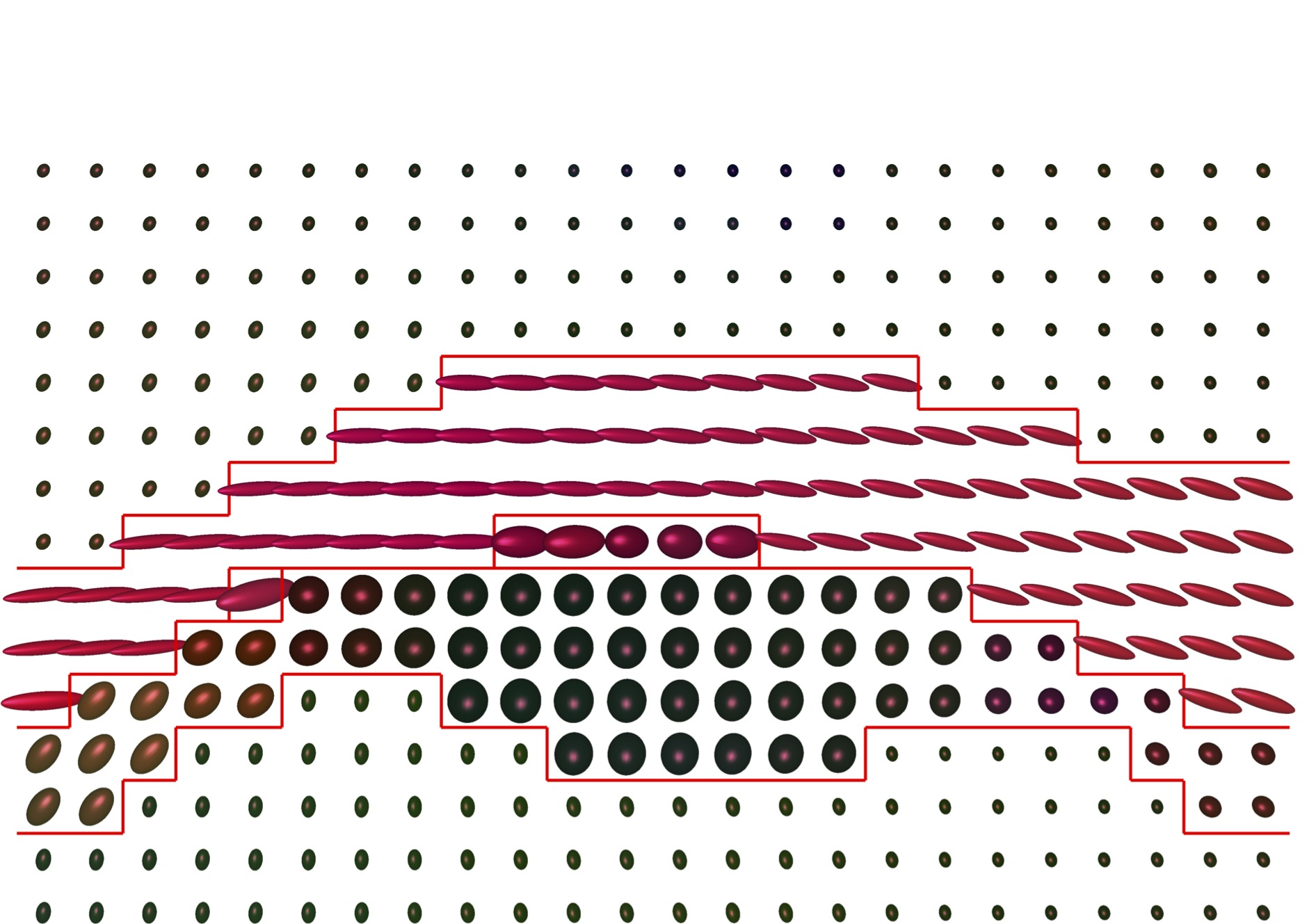}}
    \caption{}
  \end{subfigure}
\caption{(a) Corpus callosum of a human brain from the Camino project \cite{cook2006camino}. 
	(b) Mumford-Shah regularization ($p,q=1$) using parameters $\alpha= 4.3$ and $\gamma = 2.9$.
The noise is reduced significantly while the edges are preserved. 
In particular, the reconstruction induces a segmentation of
the corpus callosum and its adjacent structures (red lines).
}
\label{fig:corpusCallosum}
\end{figure*}

The data measured in DTI
are so-called diffusion weighted images (DWIs) $D_v$ which capture the directional diffusivity in the direction $v.$
The relation between the diffusion tensor image $f$ and the DWIs $D_v$
at some pixel $p$ is given by the Stejskal-Tanner equation
\begin{equation} \label{eq:StejTan}
   D_v(p) =  A_0 e^{- b \ v^T S(p) v},
\end{equation}
where $b,A_0>0$ are empirical parameters. For our simulation, we used $b = 800$ and $A_0 = 1000.$
The tensor $S(p)$ is commonly derived from the DWIs via a least square fit using \eqref{eq:StejTan}.
In our experiments we visualize the diffusion tensors by the isosurfaces of the corresponding quadratic forms. More precisely, the ellipse representing
the diffusion tensor $S(p)$ at pixel $p$ are the points $x$ fulfilling $(x-p)^T S(p)(x-p) = c,$ for some $c>0.$

We  simulate noisy data using a Rician noise model \cite{basu2006rician, fillard2007clinical}.
This means that we generate a noisy DWI $D'_v(p)$ by
\[
D'_v(p) = \sqrt{(X+ D_v(p))^2 + Y^2},
\]
with clean data $D_v(p)$ and Gaussian variables $X,Y \sim N(0,\sigma^2)$.
In our examples,
we impose Rician noise to $15$ diffusion weighted images and then compute
the diffusion tensors  according to the Stejskal-Tanner equation \eqref{eq:StejTan} using a least squares fit.
We compare our results with $L^p$-$V^q$ regularization,
i.e., with minimizers of the two-dimensional analogue of \eqref{eq:minimization_problem}
using the (globally convergent) cyclic proximal point algorithm of \cite{weinmann2013total}.
We optimized the model parameter with
respect to the error to the groundtruth.

The univariate situation is illustrated in Fig.~\ref{fig:ex2_1D} for Potts 
and in Fig.~\ref{fig:ex3_1D}
for Mumford-Shah regularization.
Fig.~\ref{fig:synthDTI2D}
shows the effect of Potts regularization
on a simple diffusion tensor image.
The noise is removed and the segment boundaries are correctly recovered.
The image in Fig.~\ref{fig:synthCC} possesses a certain variation within the segments.
Therefore the (piecewise smooth)
Mumford-Shah regularization is the proper method.
As result, we obtain a piecewise smooth denoised image with preserved sharp edges.

\subsection{Application to real data --  segmentation of the corpus callosum}
DTI is frequently used to study characteristics of the corpus callosum.
The corpus callosum
connects the right and the left hemisphere of the human brain.
Typically, the first step of
an analysis is the localization of the corpus callosum
\cite{wang2005dti, alexander2007diffusion}.
We use our Mumford-Shah method 
for the segmentation of the corpus callosum of a human brain.
This real data set stems from the Camino project \cite{cook2006camino}.
In  Fig.~\ref{fig:corpusCallosum}, we observe that our Mumford-Shah approach
removes noise and preserves
sharp boundaries between the oriented structures.
In particular, the jump set yields an accurate segmentation of the
corpus callosum.

%% file: pottsManifold_proofs.tex

\section{Existence of minimizers}\label{secAppendix:ExMinimizers}

We supply the proofs of Theorem~\ref{thm:ExMinimizers2D} and Theorem~\ref{thm:ExMinim1D} 
which are statements on the existence of minimizers.

\begin{proof}[Proof of Theorem~\ref{thm:ExMinimizers2D}]
	
	We first show that the Mumford-Shah version of the discretization \eqref{eq:msDiscrete} has a minimizer. 
	In the Mumford-Shah case, $\psi$ is the truncated power function given by \eqref{eq:psiMumfi}. 
	Since $\psi$ is continuous, so is $\Psi_{a_s}$ for all $s$ 
	and therefore the whole functional given by \eqref{eq:msDiscrete} is continuous.
	On the other hand, the data term $d^p(x, f)$ is obviously coercive with respect to the Riemannian distance.
	This makes the overall functional coercive and confines points with small functional value to a bounded set.
	Since the manifold under consideration is complete, points with small functional value are confined to a compact set.
	Hence, the continuous functional takes its minimal value on this compact set and the corresponding point is a minimizer.
	
	We come to the discrete Potts functional. Here we consider the discretization \eqref{eq:msDiscrete} where 
	$\psi$ is implemented by \eqref{eq:psiPotts}. With the same argument as for the  Mumford-Shah functional above, the 
	Potts functional is coercive with respect to the Riemannian distance. We show its lower semicontinuity.
	We have a look at $\Psi_{a_s}$ which can be written as a sum of univariate jump functionals for manifold-valued data
	of the form $S: u \mapsto |\mathcal J(u)|$ from the Riemannian manifold $M^j$ to the nonnegative integers 
	(where $j$ is the varying length of the data under consideration.) If these functionals $S$ were not lower semicontinuous,
	there would be a convergent sequence $u^n \to u$ with each $u^n \in M^j$ such that $|J(u)|>|J(u^n)|$ for sufficiently high indices $n.$
    Since $u^n \to u$ componentwise (with respect to the distance induced by the  Riemannian metric), 
    we get, using the triangle inequality, that
    $$d(u^n_k;u^n_{k-1}) \to d(u_k;u_{k-1}).$$ 
    This contradicts $u$ having more jumps than $u^n.$ Hence, the functionals $S$ and, as a consequence,
	the functionals $\Psi_{a_s}$ are lower semicontinuous. Using the continuity of the data term the discretization \eqref{eq:msDiscrete}
	of the Potts functional is lower semicontinuous. By its coercivity and the completeness of the manifold $M,$ arguments with a small 
    Potts value are located in a compact set. Hence, in the Potts case, \eqref{eq:msDiscrete} has a minimizer.
    This completes the proof.
\end{proof}

\begin{proof}[Proof of Theorem~\ref{thm:ExMinim1D}]
	The assertion is a consequence of Theorem~\ref{thm:ExMinimizers2D} 
	when specifying to data defined on $\{1,\ldots,n\} \times \{1\}$ choosing as single direction $a_1=(1,0)$.  	
\end{proof}

\section{Univariate Mumford-Shah and Potts al\-go\-rithms}
\label{secAppend:UnivarMinimAlgo}

We supply the proof of Theorem~\ref{thm:AlgProducesMinimizers} which states that the algorithms proposed for the univariate problems produce global minimizers when the data live in a Cartan-Hadamard manifold.

\begin{proof}[Proof of Theorem~\ref{thm:AlgProducesMinimizers}]
	We start with the Mumford-Shah problem for manifold-valued data.
	For $l=1, \hdots, r$, we consider the first $l-1$ data items  $f_{1:l-1} = (f_1, \ldots, f_{l-1})$. 
	We let $x^{l-1}$ be a minimizer of the corresponding functional $B_{\alpha,\gamma}^{l-1}$ for the truncated data  $f_{1:l-1}$.
	Moreover, we let $h^{l,r} \in M^{r-l+1}$ be the result 
	computed by our algorithm for the minimization of $V_\alpha$ according to Section \ref{subsec:AlgMumf1d}
	for data $f_{l:r}$.
	Since we are in a Cartan-Hadamard manifold, $h^{l,r}$ is a global minimizer of $V_\alpha$ 
	by Theorem~2 in \cite{weinmann2013total}.	
	With each $l$ we associate the candidate $x^{l,r}=(x^{l-1},h^{l, r}).$
	On the other hand we consider an index $l^\ast$ minimizing \eqref{eq:potts_value_candidate}.
	We claim that the candidate $x^{l^\ast,r}$ is a minimizer of $B_{\alpha,\gamma}^{r}.$
	To see this, consider an arbitrary $x \in M^r$ and let $k$ be its rightmost jump point $k$.
	If there is no such $k,$ then $x$ has no jumps and 
	$$
	B_{\alpha,\gamma}^{r}(x) = V_\alpha(x) \geq V_\alpha(x^{1,r}) \geq B_{\alpha,\gamma}^{r}(x^{l^\ast,r}).
	$$
	The penultimate inequality is due to the fact that $x^{1,r}$ is a global minimizer of $V_\alpha$ in a Cartan-Hadamard manifold.
	The last inequality follows from \eqref{eq:potts_value_candidate}.  
	If $k$ is the rightmost jump point of $x,$ we have 
	$$
	B_{\alpha,\gamma}^{r}(x) = B_{\alpha,\gamma}^{k-1}(x)+\gamma+V_\alpha(x^{l,r})  \geq B_{\alpha,\gamma}^{r}(x^{l^\ast,r})
	$$
	by \eqref{eq:potts_value_candidate}.	
	This shows the assertion of the theorem in the Mumford-Shah case using induction on $r$.	
	
	In the Potts functional case, 
    we let $x^{l-1}$ be a minimizer of the Potts functionals $P_{\gamma}^{l-1}$ for the truncated data $f_{1:l-1}$.
    Then we let $h^{l,r} \in M^{r-l+1}$ be the result of the gradient (resp. subgradient) descent \eqref{eq:GradientDescentIntMean}.
    Since we are in a Cartan-Hadamard manifold, $h^{l,r}$ agrees with the constant function on $[l,r]$
    which is pointwise equal to the mean $(p=2),$ median $(p=1)$ or, in general, the minimizer of the right hand side of \eqref{eq:epsilonPotts}.
    Now we may proceed analogous to the Mumford-Shah case to conclude the assertion and complete the proof. 	
\end{proof}

We proceed showing Theorem \ref{thm:AlgProducesMinimizersAdmissPartitions} which states that our algorithm yields a minimizer
for the Potts problem when considering general complete Riemannian manifolds and candidates with admissible partitions. 

\begin{proof}[Proof of Theorem~\ref{thm:AlgProducesMinimizersAdmissPartitions}]
    We use the notation of the proof of Theorem~\ref{thm:AlgProducesMinimizers}.
    Then, the $x^{l-1}$ are minimizer of the corresponding Potts functionals $P_{\gamma}^{l-1}$ for the truncated data $f_{1:l-1}$.
    (We notice that such a minimizer exists, since an interval consisting of one member is always admissible.)
    Furthermore, for admissible intervals $[l,r]$, $h^{l,r} \in M^{r-l+1}$ is pointwise equal to the computed Riemannian mean 
    as explained in Section \ref{subsec:AlgPotts1d}. 
    The Riemannian mean minimizes the right hand side of \eqref{eq:epsilonPotts}.
    The candidates $x^{l,r}=(x^{l-1},h^{l, r})$ and the minimizing index 
    $l^\ast$ are given as in the proof of Theorem~\ref{thm:AlgProducesMinimizers} above.
    In order to show that $x^{l^\ast,r}$ is a minimizer,
    we consider an arbitrary $x \in M^r$ with an admissible partition.
    If $x$ has no jump, then
    $ P_{\gamma}(x) = \tfrac{1}{2}\sum_i d(x,f_i)^2 \geq$ $ P_{\gamma}(x^{1,r})  \geq  P_{\gamma}(x^{l^\ast,r}).$
    Otherwise, let $k$ be the rightmost jump point of $x$ (which, by assumption, comes with an admissible partition).
    Then,
    $$
    P_{\gamma}^{r}(x) = P_{\gamma}^{k-1}(x)+\gamma+V_\alpha(x^{l,r})  \geq P_{\gamma}^{r}(x^{l^\ast,r}).
    $$
    which shows that $x^{l^\ast, r}$ is a minimizer. Now induction completes the proof.
\end{proof}

\section{Mumford-Shah and Potts algorithms for images}
\label{secAppend:Convergence2D}

We supply the proof of Theorem~\ref{thm:Convergence2D} stating that the algorithm in \eqref{eq:OurSplitting} converges in a Cartan-Hadamard manifold.

\begin{proof}[Proof of Theorem~\ref{thm:Convergence2D}]	
	We show that all iterates $x^k_s$ converge to the same limit for all $s \in \{1,\ldots,R\}.$	
	Since we are in a Cartan-Hadamard manifold, $x_1^{k+1}$ is a global minimizer of the functional 
	$H_1(x) = pR\omega_1 \alpha \Psi_{a_1}(x) + d^p(x,f) + \mu_k d^p(x,x_R^k) $ which is the first problem in \eqref{eq:OurSplitting}.	
	This follows by an argument similar to the proof of Theorem~\ref{thm:AlgProducesMinimizers}.
	
	We have $H_1(x_1^{k+1}) \leq H_1(x_R^k)$ which means that 
	\begin{align}
	d^p(x^{k+1}_1, f) + \mu_k d^p(x^{k+1}_1,x_{R}^k) 
	\leq pR\omega_1 \alpha\Psi_{a_1}(x^{k}_R) + d^p(x^{k}_R, f).
	\label{eq:No1}
	\end{align}
	In analogy, we get for the  $x_s^{k+1},$ $s=2,\ldots,R$, using the other functionals in \eqref{eq:OurSplitting} that
	\begin{align}
	d^p(x^{k+1}_s, f) + \mu_k d^p(x^{k+1}_s,x_{s-1}^k) 
	\leq pR\omega_1 \alpha \Psi_{a_s}(x^{k+1}_{s-1}) + d^p(x^{k+1}_{s-1}, f).\label{eq:No2}
	\end{align}
	For both the Mumford-Shah and the Potts problem, the terms $\alpha\Psi_{a_1}(x^{k}_R)$ and $\alpha\Psi_{a_s}(x^{k+1}_{s-1}),$ 
	with $s=2,\ldots,R,$ are uniformly bounded by a constant $C$ which does not depend on $k$ and $s.$
	This is because, for any input, $\alpha \Psi_{a_s}$ is bounded by $\alpha mn$ with the regularizing parameter  $\alpha$ for the jump term of the functional under consideration, and $m$ and $n$ are the height and width of the image. 
	Hence we can use \eqref{eq:No1} and \eqref{eq:No2} to get 
	\begin{align}
	 &d^p(x^{k+1}_1,x_{R}^k) \leq  \tfrac{C}{\mu_{k}} + \tfrac{1}{\mu_{k}}(d^p(x^{k}_R, f) - d^p(x^{k+1}_1, f)),\notag \\	 	
	 &d^p(x^{k+1}_s,x_{s-1}^k) \leq \tfrac{C}{\mu_{k}}+ \tfrac{1}{\mu_{k}}(d^p(x^{k+1}_{s-1}, f) - d^p(x^{k+1}_s, f)) .
	\label{eq:No3}
	\end{align}
	Now we may apply the inverse triangle inequality to the second summand on the right-hand side and get 
	$d^p(x^{k}_R, f) - d^p(x^{k+1}_1, f) \leq d^p(x^{k}_R,x^{k+1}_1).$ Then, a simple manipulation shows that 
	\begin{align}
	d^p(x^{k+1}_1,x_{R}^k) \leq  \tfrac{C}{\mu_{k}-1}, 	
	\quad d^p(x^{k+1}_s,x_{s-1}^k) \leq \tfrac{C}{\mu_{k}-1}.
	\label{eq:No4}
	\end{align}
	As a consequence, there is a constant $D$ and an index $k_0$ such that, for all $k\geq k_0,$
	\[
	d(x^{k+1}_R,x_{R}^k) \leq D \mu_{k}^{-\tfrac{1}{p}}.	
	\]  
	Hence,
	\[
	d(x^{k+1}_R,x_{R}^{k_0}) \leq D \sum_{l=k_0+1}^{k+1} \mu_{l}^{-\tfrac{1}{p}} < \infty,	
	\]
	and so the sequence $x^{k+1}_R$ converges. By \eqref{eq:No3}, the  iterates $x_{s}^k$ converge to the same limit for all $s=1,\ldots,R-1.$
	This completes the proof.	
\end{proof}